\documentclass[a4paper,12pt, twoside]{article}

\newtheorem{thm}{Theorem}[section]
\newtheorem{prop}{Proposition}[section]
\newtheorem{rmk}{Remark}[section]
\newtheorem{lem}{Lemma}[section]
\newtheorem{defi}{Definition}[section]

\usepackage{authblk}
\usepackage{amsmath, amssymb}
\usepackage{amstext}
\usepackage[english]{babel}
\usepackage{xcolor}
\usepackage{amsfonts}
\usepackage{stackengine}
\usepackage{graphicx}
\usepackage{mathtools}
\usepackage{newtxtext}
\usepackage{times}
\usepackage{urwchancal}
\usepackage{bbold}
\usepackage{hyperref}
\usepackage{cleveref}
\usepackage{centernot}
\usepackage{subcaption}
\usepackage[export]{adjustbox}
\usepackage{mathrsfs}
\usepackage[md]{titlesec}
\usepackage[style=numeric-comp, backend=bibtex]{biblatex}
\usepackage{fancyhdr}
\usepackage{enumitem}
\usepackage{indentfirst}
\usepackage{helvet} 

\newenvironment{proof}{\paragraph{Proof:}}{\hfill$\square$ \\}

\numberwithin{equation}{section}

\allowdisplaybreaks

\providecommand{\keywords}[1]
{
	\small	
	\textbf{\textit{Keywords:}} #1
}

\providecommand{\MSC}[1]
{
	\small	
	\textbf{\textit{2020 MSC:}} #1
}

\fancypagestyle{mypagestyle}{
	\fancyhf{}
	\fancyhead[OC]{\textit{J. Doghman} \& \textit{L. Gouden{\`e}ge}}
	\fancyhead[EC]{\textit{Convergence analysis of the stochastic NS-$\alpha$ toward the NS equations}}
	\fancyhead[LO, LE]{\thepage}
	\fancyfoot{}
	
}
\title{Convergence of the stochastic Navier-Stokes-$\alpha$ solutions toward the stochastic Navier-Stokes solutions}
\author[ ]{Jad Doghman \thanks{Corresponding author \newline Jad Doghman is supported by a public grant as part of the Investissement d'avenir project [ANR-11-LABX-0056-LMH, LabEx LMH], and both authors are part of the SIMALIN project [ANR-19-CE40-0016] of the French National Research Agency. \newline Email addresses: \texttt{jad.doghman@centralesupelec.fr} (J. Doghman),\newline {\mbox{ }\hspace{67pt}} \texttt{goudenege@math.cnrs.fr} (L. Gouden{\`e}ge)}}
\author[ ]{Ludovic Gouden{\`e}ge}
\affil[ ]{CNRS, F{\'e}d{\'e}ration de Math{\'e}matiques de CentraleSup{\'e}lec FR 3487, Universit{\'e} Paris-Saclay, CentraleSup{\'e}lec, 91190 Gif-sur-Yvette, France}
\date{}
\begin{document}
	\maketitle
	\begin{abstract}
		Loosely speaking, the Navier-Stokes-$\alpha$ model and the Navier-Stokes equations differ by a spatial filtration parametrized by a scale denoted $\alpha$. Starting from a strong two-dimensional solution to the Navier-Stokes-$\alpha$ model driven by a multiplicative noise, we demonstrate that it generates a strong solution to the stochastic Navier-Stokes equations under the condition $\alpha \to 0$. The initially introduced probability space and the Wiener process are maintained throughout the investigation, thanks to a local monotonicity property that abolishes the use of Skorokhod's theorem. High spatial regularity a priori estimates for the fluid velocity vector field are carried out within periodic boundary conditions.
	\end{abstract}
	
	\noindent\keywords{Navier-Stokes-$\alpha$, Navier-Stokes, multiplicative noise, cylindrical Wiener process, strong solutions}\newline
	\MSC{60H15, 60H30, 37L55, 35Q30, 35Q35, 76D05}
	
	\section{Introduction}
	To circumvent most of the Navier-Stokes drawbacks, a reasonable amount of Large Eddy Simulation (LES) models have been created and introduced to the fluid mechanics' literature. Among them is the Navier-Stokes-$\alpha$ (NS-$\alpha$) model, which made its appearance in \cite{Chen1999Foias, HolmMarsdenRatiu} and is known under the names: Lagrangian averaged Navier-Stokes (LANS-$\alpha$) equations \cite{Marden2001Shkoller} or the viscous Camassa-Holm problem \cite{Bjorland2008Schonbek}. Given a solution to the stochastic NS-$\alpha$ model:
	\begin{equation}\label{main equation}
		\begin{cases}
			\begin{aligned}
				&\frac{\partial}{\partial t}\left(\bar{u} -\alpha^{2}\Delta \bar{u}\right) - \nu\Delta\left(\bar{u} - \alpha^{2}\Delta \bar{u}\right) - \bar{u}\times\left(\nabla \times \left(\bar{u} - \alpha^{2}\Delta \bar{u}\right)\right) + \nabla \bar{p} = g(\cdot, \bar{u})\frac{\partial W}{\partial t}, 
			\end{aligned}\\
			div(\bar{u}) = 0, \\
			\bar{u}(0, \cdot) = \bar{u}_{0},
		\end{cases}
	\end{equation}
	the main interest in this paper is to check whether or not it converges toward a solution of the stochastic Navier-Stokes equations (NSEs) 
	\begin{equation}\label{eq NS}
		\begin{cases}
			\frac{\partial u}{\partial t} + -\nu\Delta u + [u\cdot\nabla]u + \nabla p = g(\cdot, u)\frac{\partial W}{\partial t}, \\
			div (u) = 0, \\
			u(0, \cdot) = \bar{u}_{0},
		\end{cases}
	\end{equation}
	when the spatial scale $\alpha$ tends to $0$. Both equations are equipped with the same configurations, including the initial datum $\bar{u}_{0}$ to guarantee a similar fluid state at time $t = 0$. The two-dimensional vectors $\bar{u}$ and $u$ denote the fluid velocities, the $\mathbb{R}$-valued quantities $p$ and $\bar{p}$ represent the pressure fields, the positive constant $\nu$ symbolizes the kinematic fluid viscosity, $\alpha$ is a small positive spatial scale at which the fluid motion is filtered, $g$ is a diffusion coefficient depending on the velocity vector field, and $W$ is an infinite-dimensional (possibly cylindrical) Wiener process. On account of the poor uniqueness properties of three-dimensional solutions to the stochastic NSEs, the conducted study herein will be limited to two dimensions to guarantee that the unique solution of the stochastic NS-$\alpha$ equations converges toward a sole one as $\alpha$ goes to $0$.
	
	In this paper, the study is accomplished through periodic boundary conditions for the sake of investigating the effect of $\alpha$ on the space regularity of a solution and taking advantage of the nonlinearity's properties that occur within this framework. It could have been carried out within Dirichlet boundary conditions if only the typical solution's space regularity was intended. Observe that $\alpha$ is always multiplied by $\Delta \bar{u}$ in equations~\eqref{main equation}, meaning that the extra granted regularity that does not figure in problem~\eqref{eq NS} can be loosened through a particular assumption on $\alpha$ when dealing with a finite-dimensional system, namely a Faedo-Galerkin approximation. The pressure field will be eliminated from the corresponding weak formulation throughout this work through the null divergence criterion, and the focus will be turned toward the velocity vector. Equations~\eqref{main equation} will be transformed into a coupled problem of second-order so that its form matches somehow that of system~\eqref{eq NS}, and the spatial scale $\alpha$ will be controlled by the inverse of a specific eigenvalue of the Stokes operator for the sake of absorbing the extra space regularity that is delivered by equations~\eqref{main equation}.
	
	Investigating the convergence of equations~\eqref{main equation} toward system~\eqref{eq NS} is beneficial because the principal reason for which the NS-$\alpha$ model was introduced is to overcome most of the Navier-Stokes shortcomings. If the converse scenario took place, equations~\eqref{main equation} would have become obsolete, but fortunately, it is not. This convergence was also conducted for the deterministic settings (i.e. when $g = 0$) in \cite{Cao2009Titi}, where the convergence rate in terms of $\alpha$ is revealed. The theoretical study herein has the advantage of building efficient numerical schemes for the stochastic Navier-Stokes problem while considering minimal assumptions on the spatial scale $\alpha$. Since $\alpha$ is solely involved with solutions' space regularity, any time discretization should not come into play in any further hypotheses upon $\alpha$.
	
	Equations~\eqref{main equation} were first inspected in \cite{Caraballo2006Takeshi, caraballo2005stochastic}, where the existence of a unique variational solution was proven. It is worth highlighting one drawback of this model relative to the pressure's regularity that appears after applying a generalization of the De Rham theorem \cite{pressure2003}, which links the velocity's smoothness to that of the pressure. In point of fact, it was shown (c.f. \cite[Theorem 3.3]{Caraballo2006Takeshi}) that $\bar{p}$ is $H^{-1}$-valued, meaning that it is lower than that of $p$, which is $L^{2}$-valued. This inconvenience originates from the biharmonic operator that appears in the first identity of system~\eqref{main equation} and might have an uncooperative effect on convergence rates of numerical schemes concerned with a non-null divergence of velocities. The same goes for other stochastic Navier-Stokes variants, such as the Leray-$\alpha$ model \cite{Deugoue2010Sango}. Further examinations of equations~\eqref{main equation} were performed in \cite{Deugoue2009Sango, deugoue2011weak}, including a splitting-up scheme in \cite{deugoue2014convergence}.
	
	This paper is organized as follows: all preliminaries, assumptions and configurations are presented in Section~\ref{section conf and material}, which allows the main theorem of this work to be stated in Section~\ref{section main result}, followed by Section~\ref{section Faedo-Galerkin} where the Faedo-Galerkin approximation of equations~\eqref{main equation} is exploited to acquire a finite-dimensional system, and a priori estimates are carried out within multiple spatial regularities. Section~\ref{section convergence} provides the convergence steps of the projected system, including the local monotonicity property, which is a prominent member of the demonstration. Finally, a conclusion regarding the accomplished analysis in the previous section, the relationship with the Navier-Stokes problem, and a few perspectives are given in Section~\ref{section conclusion}.
	
	\section{Configuration and materials}\label{section conf and material}
	Given a positive number $L$, the domain $D$ represents a two-dimensional torus $(0, L)^{2}$, and for a given $T>0$, the time interval reads $[0,T]$. Throughout this paper, the Lebesgue and Sobolev spaces are denoted $L^{p}$ and $H^{m}$ (or $W^{m, p}$) respectively, and for an arbitrary normed vector space $X$, its associated norm will be symbolized by $\left|\left|\cdot\right|\right|_{X}$. The notation $X_{per}$ signifies that all its members are periodic functions whose mean is null. Regarding the small spatial scale $\alpha$ that is present in equations~\eqref{main equation}, a special norm $\left|\left|\cdot\right|\right|_{\alpha}$ is associated with it and defined by $\left|\left|\cdot\right|\right|^{2}_{\alpha} \coloneqq \left|\left|\cdot\right|\right|^{2}_{L^{2}} + \alpha^{2}\left|\left|\nabla\cdot\right|\right|^{2}_{L^{2}}$. The notation $\mathscr{L}_{2}(E, F)$ is the space of all Hilbert-Schmidt operators; with $E$ and $F$ being two given Banach spaces, $\lesssim$ embodies a shorthand for the less or equal symbol $\leq$ up to a universal non-negative constant, and $C_{D}$ will denote throughout this paper a positive constant depending only on the domain $D$. The solely employed Gelfand triple herein is $\left(H^{1}_{per}(D), L_{per}^{2}(D), H_{per}^{-1}(D)\right)$, where $H_{per}^{-1}(D)$ is the dual space of $H^{1}_{per}(D)$. The $L^{2}(D)$ space will be endowed with its standard inner product $\left(\cdot, \cdot\right)$, and the duality brackets $\langle \cdot, \cdot \rangle$ will represent the duality product between $H^{1}_{per}(D)$ and $H_{per}^{-1}(D)$.
	Following the mathematical notations for the Navier-Stokes framework, the function spaces that will be frequently encountered herein are
	\begin{align*}
		&\mathcal{V} \coloneqq \left\{u \in [C^{\infty}_{per}(D)]^{2} \ \big| \ div(u) = 0\right\},\\&
		\mathbb{H} \coloneqq \left\{ u \in [L_{per}^{2}(D)]^{2} \ \big| \ div(u) = 0 \mbox{ a.e. in } D\right\},\\&
		\mathbb{V} \coloneqq \left\{u \in [H_{per}^{1}(D)]^{2} \ \big| \ div(u) = 0 \mbox{ a.e. in } D\right\}.
	\end{align*}
	Let $A$ be the Stokes operator defined from $D(A) \coloneqq [H^{2}(D)]^{2}\cap \mathbb{V}$ into $\mathbb{H}$ by $A \coloneqq -\mathcal{P}\Delta$, where $\mathcal{P}\colon [L_{per}^{2}(D)]^{2} \to \mathbb{H}$ is the Leray Projector. In two-dimensional domains and under periodic boundary conditions, it is well-known that the Laplace-Leray commutator $[\mathcal{P}, \Delta]$ vanishes; namely $\mathcal{P}\Delta = \Delta\mathcal{P}$. Recall that operator $A$ is self-adjoint whose inverse is compact (c.f. \cite{Peter1988, temam2001navier}).
	From now on, all Cartesian products of a sole linear space will be symbolized by blackboard bold letters with the domain $D$ being omitted. For instance, the Sobolev space $[H_{per}^{1}(D)]^{2}$ will become $\mathbb{H}_{per}^{1}$.
	
	Let $\left(\Omega, \mathcal{F}, (\mathcal{F}_{t})_{0 \leq t \leq T}, \mathbb{P}\right)$ be a filtered complete probability space whose filtration $(\mathcal{F}_{t})_{0 \leq t \leq T}$ is right-continuous. Given a separable Hilbert space $K$ equipped with a complete orthonormal basis $\{w_{k}, k \geq 1\}$, the $K$-valued cylindrical Wiener process $W(t), t\in [0,T]$ reads 
	\begin{equation*}
		W(t) \coloneqq \sum_{k\geq 1}\beta_{k}(t)w_{k}, \ \ \forall t \in [0,T],
	\end{equation*}
	where $\{\beta_{k}, k \geq 1\}$ is a family of independent and identically distributed $\mathbb{R}$-valued Brownian motions on $\left(\Omega, \mathcal{F}, (\mathcal{F}_{t})_{0 \leq t \leq T}, \mathbb{P}\right)$. For any $\phi \in L^{2}\left(\Omega; L^{2}(0,T; \mathscr{L}_{2}(K, \mathbb{L}^{2}))\right)$, its stochastic integral with respect to the Wiener process $\left\{W(t), t \in [0,T] \right\}$ is defined (c.f. \cite{prevot2007concise}) as the unique continuous $\mathbb{L}^{2}$-valued $\mathcal{F}_{t}$-martingale such that for all $\psi \in \mathbb{L}^{2}$,
	\begin{equation*}
		\left(\int_{0}^{t}\phi(s)dW(s), \psi\right) = \sum_{k\geq 1} \int_{0}^{t}\left(\phi(s)w_{k}, \psi\right)d\beta_{k}(s), \ \ \forall t \in [0,T].
	\end{equation*}
	For clarity's sake, the nonlinear term in equations~\eqref{main equation} will be denoted $\tilde{b}$; that is 
	\begin{equation*}
		\tilde{b}(u, v, w) = -\Big(u \times (\nabla\times v), w\Big)
	\end{equation*}
	for appropriate vector fields $u$, $v$ and $w$, where $v = u - \alpha^{2}\Delta u$ in the equations of interest. The bilinear operator that can be derived from $\tilde{b}$ will be denoted $\tilde{B}$ and it reads: $\displaystyle \tilde{B}(u, v) \coloneqq -u \times (\nabla\times v)$, for all $u, v \in \mathbb{V}$.
	The below proposition lists a few useful properties of the bilinear operator $\tilde{B}$.
	\begin{prop}\label{prop trilinear term}
		The following assertions are satisfied by the nonlinear term:
		\begin{enumerate}[label=(\roman*)]
			\item For all $u, v, w \in \mathbb{H}^{1}$, $\langle \tilde{B}(u,v),w \rangle = - \langle \tilde{B}(w, v), u\rangle$. In particular, $\langle \tilde{B}(u, v), u\rangle = 0$.
			\item $\langle \tilde{B}(u,v), w \rangle = \left([u\cdot\nabla]v, w\right) - \left([w\cdot\nabla]v, u\right)$, for all $u,v,w \in \mathbb{H}_{per}^{1}$. If additionally, $u$ and $v$ are divergence-free then, $\langle \tilde{B}(u,v), v \rangle = - \left([v\cdot\nabla]v, u\right)$.
			\item $\left|\langle \tilde{B}(u, v), w \rangle\right| \leq C_{D}\left|\left|u\right|\right|_{\mathbb{L}^{4}}\left|\left|\nabla v\right|\right|_{\mathbb{L}^{2}}\left|\left|w\right|\right|_{\mathbb{L}^{2}}^{\frac{1}{2}}\left|\left|\nabla w\right|\right|_{\mathbb{L}^{2}}^{\frac{1}{2}}$, for all $u, v, w \in \mathbb{H}^{1}_{per}$.
		\end{enumerate}
	\end{prop}
	\begin{proof}
		Assertion \textit{(i)} can be proven by a simple application of the identity $\displaystyle \left(u \times v\right)\cdot w = -(w \times v)\cdot u$. To demonstrate equality \textit{(ii)}, we need to employ the following property:
		\begin{equation}\label{eq calc0}
			\langle \tilde{B}(u, v), w\rangle = \left([u\cdot\nabla]v, w\right) + \left((\nabla u)^{T}\cdot v, w\right) - \left(\nabla (u\cdot v), w\right),
		\end{equation}
		which may be straightforwardly proven via the identity
		\begin{equation*}
			[u\cdot \nabla]v + (\nabla u)^{T}\cdot v - \nabla(u \cdot v) = -u\times (\nabla \times v).
		\end{equation*}
		Indeed, the quantity $\displaystyle \left((\nabla u)^{T}\cdot v, w\right)$ of equation~\eqref{eq calc0} turns into $\displaystyle -\left([w\cdot \nabla]v, u\right) + \left(\nabla (u\cdot v), w\right)$ after applying two consecutive integration by parts. Plugging it back in equation~\eqref{eq calc0} completes the proof of \textit{(ii)}. Finally, the H{\"o}lder and Ladyzhenskaya (see \cite[Lemma I.1]{Ladyzhenskaya1964}) inequalities applied to assertion \textit{(ii)} yield estimate \textit{(iii)}.
	\end{proof}
	
	The operator $\tilde{b}$ can be readily expressed via the trilinear form associated with the Navier-Stokes equations, as mentioned in Proposition~\ref{prop trilinear term}-\textit{(ii)}. For brevity's sake, we deploy the next proposition to grant a few corresponding properties. The reader may refer to \cite[Remark 2.2]{temam1995navier} for further information.
	\begin{prop}\label{prop trilinear term NS}
		\begin{enumerate}[label=(\roman*)]
			\item $\displaystyle \left([u\cdot\nabla]v, v\right) = 0$ for all $u, v \in \mathbb{V}$.
			\item $\displaystyle \left|\left([u\cdot\nabla]v, w\right)\right| \leq C_{D}\left|\left|u\right|\right|_{\mathbb{L}^{2}}\left|\left|\nabla v\right|\right|_{\mathbb{L}^{2}}\left|\left|w\right|\right|^{\frac{1}{2}}_{\mathbb{L}^{2}}\left|\left|Aw\right|\right|^{\frac{1}{2}}_{\mathbb{L}^{2}}$, for all $u \in \mathbb{H}$, $v \in \mathbb{V}$ and $w \in D(A)$.
		\end{enumerate}
	\end{prop}
	
	\paragraph{\textbf{Assumptions}}
	\begin{enumerate}[label=$(S_{\arabic*})$]
		\item $\displaystyle \mathbb{E}\left[\left|\left|\bar{u}_{0}\right|\right|_{\mathbb{H}^{1}}^{2^{p}}\right] < +\infty$, for some $p \in [1, +\infty)$, \label{S1}
		\item  $g \in L^{2}\left(\Omega; L^{2}(0,T; \mathscr{L}_{2}(K, \mathbb{L}^{2}))\right)$ satisfies: for all $u \in \mathbb{V}$, $g(\cdot ,u)$ is $\mathcal{F}_{t}$-progressively measurable, and almost everywhere in $\Omega \times (0,T)$, it holds that: 
		\begin{equation*}
			\begin{aligned}
				&\lvert|\lvert g(\cdot, u) - g(\cdot, v)\rvert\rvert_{\mathscr{L}_{2}(K, \mathbb{L}^{2})} \leq L_{g}\lvert\lvert u - v\rvert\rvert_{\alpha}, \ \ \forall u, v \in \mathbb{V}, \\&
				\lvert\lvert g(\cdot, u)\rvert\rvert_{\mathscr{L}_{2}(K, \mathbb{H}^{1})} \leq K_{1} + K_{2}\left|\left|u\right|\right|_{\alpha}, \ \ \forall u \in \mathbb{V}.
			\end{aligned}
		\end{equation*}
		for some real, nonnegative, time-independent constants $L_{g}$, $K_{1}, K_{2}$.\label{S2}
	\end{enumerate}
	\begin{rmk}
		Inequality $\lvert\lvert g(\cdot, u) \rvert\rvert_{\mathscr{L}_{2}(K, \mathbb{H}^{1})} \leq K_{1} + K_{2}\lvert\lvert u\rvert\rvert_{\alpha}$	of assumption~\ref{S2} is imposed in $\mathbb{H}^{1}$ instead of $\mathbb{L}^{2}$ to be able to execute high space-regularity estimates for the velocity field.
	\end{rmk}
	To reduce repetitions, the below proposition gathers a few properties that will be employed throughout this paper.
	\begin{prop}\label{prop inequalities}
		\begin{enumerate}[label=(\roman*)]
			\item $\displaystyle x^{p} \leq 1 + x^{q}$ for all $x \geq 0$, and $1 \leq p \leq q < +\infty$.
			\item $\displaystyle 2\left(a, b\right) = \left|\left|a\right|\right|^{2}_{\mathbb{L}^{2}} - \left|\left|b\right|\right|^{2}_{\mathbb{L}^{2}} + \left|\left|a - b\right|\right|^{2}_{\mathbb{L}^{2}}$, for all $a, b \in \mathbb{L}^{2}$.
			\item $\displaystyle \left|a + b\right|^{p} \leq 2^{p-1}\left(\left|a\right|^{p} + \left|b\right|^{p}\right)$, for all $a, b \in \mathbb{R}$ and $p \geq 1$.
		\end{enumerate}
	\end{prop}
	
	\subsection{Concept of solutions}
	The underlying equations consist of a fourth-order problem which might not be insightful. Therefore, a continuous differential filter shall be introduced allowing equations~\eqref{main equation} to turn into a second-order coupled problem. 
	\begin{defi}[Continuous differential filter]
		Let $v \in \mathbb{L}^{2}$ be a given vector field. A continuous differential filter $\bar{u}$ of $v$ is defined as part of the unique solution $\left(\bar{u}, \bar{p}\right) \in \mathbb{V} \times L_{0}^{2}(D)$ to the problem.
		
		\begin{equation}\label{continuous filter}
			\begin{cases}
				\begin{aligned}
					&-\alpha^{2}\Delta \bar{u} + \bar{u} + \nabla \bar{p} = v, & \mbox{ in } D,
					\\& div (\bar{u}) = 0, & \mbox{ in } D.
				\end{aligned}
			\end{cases}
		\end{equation}
	\end{defi}
	The notation $\bar{v}$ (instead of $\bar{u}$) is widely spread in the literature of differential filters. However, to maintain a visible relationship between equations~\eqref{main equation} and \eqref{continuous filter}, $\bar{v}$ will be substituted by the notation $\bar{u}$. Observe that system~\eqref{continuous filter} represents a deterministic steady Stokes problem and that $v$ plays the role of an outer force. Additionally, projecting system~\eqref{continuous filter} using the Leray projector $\mathcal{P}$ yields
	\begin{equation*}
		\alpha^{2}A\bar{u} + \bar{u} = \mathcal{P}v, \mbox{ in } D.
	\end{equation*}
	which has a unique solution $\bar{u}$ according to \cite[Subsection 8.2]{grisvard2011elliptic}. Thereby, when it comes to the process $\left\{\bar{u}(t), t \in [0,T]\right\}$ of problem~\eqref{main equation}, the multiplication in $L^{2}$ of the above equation by $\varphi \in \mathbb{V}$ returns for all $t \in [0,T]$, 
	\begin{equation}\label{eq weak stokes}
		\left(v(t), \varphi\right) = \left(\bar{u}(t), \varphi\right) + \alpha^{2}\left(\nabla \bar{u}(t), \nabla \varphi\right).
	\end{equation}
	Based on the above identity, we define $v_{0}$ as the solution of $\left(v_{0}, \varphi\right) = \left(\bar{u}_{0}, \varphi\right) + \alpha^{2}\left(\nabla\bar{u}_{0}, \nabla \varphi\right)$, for all $\varphi \in \mathbb{V}$. Since $\bar{u}_{0}$ belongs to $\mathbb{V}$, it is straightforward that $\alpha^{2}\mathbb{E}\left[\left(\nabla\bar{u}_{0}, \nabla\varphi\right)\right] \to 0$ as $\alpha \to 0$. Subsequently, $\mathbb{E}\left[\left(v_{0}, \varphi\right)\right] = \mathbb{E}\left[\left(\bar{u}_{0}, \varphi\right)\right]$ for all $\varphi \in \mathbb{V}$ as $\alpha \to 0$. As a result, $v_{0} = \bar{u}_{0}$ $\mathbb{P}$-a.s. and a.e. in $D$ when $\alpha$ vanishes. The next definition states the compound of a solution to equations~\eqref{main equation} whose existence and uniqueness are illustrated in \cite{Caraballo2006Takeshi}.
	
	\begin{defi}\label{def variational sol}
		Let $T > 0$ and assume \ref{S1}-\ref{S2}. A $\mathbb{V} \times \mathbb{H}$-valued stochastic process $(\bar{u}(t), v(t)), \ t \in [0,T]$ is said to be a variational solution to problem \eqref{main equation} if it fulfills the following conditions:
		\begin{enumerate}[label=(\roman*)]
			\item $\displaystyle \bar{u} \in L^{2}(\Omega; L^{2}(0,T; 
			\mathbb{H}^{2}\cap\mathbb{V}) \cap L^{2}\left(\Omega; L^{\infty}(0,T; \mathbb{V})\right)$,
			\item $\displaystyle v \in L^{2}(\Omega; L^{2}(0,T; \mathbb{V})) \cap L^{2}(\Omega; L^{\infty}(0,T; \mathbb{H}))$,
			\item $\mathbb{P}$-almost surely, $\bar{u}$ is weakly continuous with values in $\mathbb{V}$, and $v$ is continuous with values in $\mathbb{H}$,
			\item for all $t \in [0,T]$, $\bar{u}$ satisfies the following equation $\mathbb{P}$-almost surely
			\begin{equation}\label{eq def sol}
				\begin{cases}
					\begin{aligned}
						&\left(v(t), \varphi\right) + \nu \int_{0}^{t}\left(\nabla v(s), \nabla\varphi\right)ds + \int_{0}^{t}\tilde{b}\left(\bar{u}(s), v(s), \varphi\right)ds \\&\hspace{20pt}= \left(v_{0}, \varphi\right) + \Big(\int_{0}^{t}g\left(s, \bar{u}(s)\right)dW(s), \varphi\Big), \ \ \forall \varphi \in \mathbb{V}, \\&
						\left(v(t), \psi\right) = \left(\bar{u}(t), \psi\right) + \alpha^{2}\left(\nabla \bar{u}(t), \nabla \psi\right), \ \ \forall \psi \in \mathbb{V}.
					\end{aligned}
				\end{cases}
			\end{equation}
		\end{enumerate}
	\end{defi}
	It is worth mentioning that the weak continuity of $\bar{u}$ is related to the strong continuity of $v$. This fact emerges from the relationship~\eqref{eq weak stokes}.
	
	Two-dimensional strong solutions to equations~\eqref{eq NS} were conducted in \cite{menaldi2002stochastic, glatt2009strong}. An appropriate definition is given by:
	\begin{defi}\label{definition NS solution}
		Let $T > 0$ be fixed and assumptions \ref{S1}-\ref{S2} be fulfilled. A process $u(t)$, $t \in [0,T]$ on a stochastic filtered probability space $\left(\Omega, \mathcal{F}, (\mathcal{F}_{t})_{t \in [0,T]}, \mathbb{P}\right)$ is said to be a strong solution to equations~\eqref{eq NS} if it belongs to $L^{2}(\Omega; C([0,T]; \mathbb{H}) \cap L^{2}(0,T; \mathbb{V}))$, and it satisfies $\mathbb{P}$-a.s. for all $t \in [0,T]$, the weak formulation
		\begin{equation*}
			\begin{aligned}
				&\left(u(t), \varphi\right) + \nu\int_{0}^{t}\left(\nabla u(s), \nabla \varphi\right)ds + \int_{0}^{t}\left([u(s)\cdot\nabla]u(s), \varphi\right)ds 
				\\&= \left(\bar{u}_{0}, \varphi\right) + \left(\int_{0}^{t}g(s, u(s))dW(s), \varphi\right), \ \ \forall \varphi \in \mathbb{V}.
			\end{aligned}
		\end{equation*}
	\end{defi}
	Equations~\eqref{eq NS} have a unique solution in the sense of Definition~\ref{definition NS solution}, see for instance \cite[Proposition 3.2]{menaldi2002stochastic}. This fact will be evoked all this paper long.
	
	\section{Main result}\label{section main result}
	\begin{thm}\label{main theorem}
		Let $T > 0$, $L > 0$, $\left(\Omega, \mathcal{F}, (\mathcal{F}_{t})_{0\leq t \leq T}, \mathbb{P}\right)$ be a filtered probability space, $D = (0, L)^{2}$ be a two-dimensional torus subject to periodic boundary conditions, and $1 \leq p < +\infty$ be given. Let $\left\{e_{k}, k \geq 1\right\}$ be a complete orthonormal basis of $\mathbb{H}$ consisting of eigenfunctions of the Stokes operator $A$, and $\left\{\mu_{k}, k \geq 1\right\}$ be the associated eigenvalues whose values diverge when $k \to +\infty$. Assume that hypotheses \ref{S1}-\ref{S2} are fulfilled, and that for all $N \in \mathbb{N}\backslash \{0\}$, the spatial scale  follows the decreasing rate $\displaystyle \mathcal{C}_{min}\mu_{N}^{-3/4} \leq \alpha \coloneqq \alpha_{N} \leq \mathcal{C}_{max}\mu_{N}^{-3/4}$, for some constants $\mathcal{C}_{min}, \mathcal{C}_{max} > 0$ independent of $N$. Then, a solution $\left(\bar{u}, v\right) \coloneqq \Big(\bar{u}(\alpha_{N}), v(\alpha_{N})\Big)$ to equations~\eqref{main equation} in the sense of Definition~\ref{def variational sol} for a given $\alpha$ converges toward the unique strong solution $v_{NS}$ of equations~\eqref{eq NS} in the sense of Definition~\ref{definition NS solution} when $N \to +\infty$, and it satisfies:
		\begin{enumerate}[label=(\roman*)]
			\item $\displaystyle \mathbb{E}\left[\sup_{t \in [0,T]}\left|\left|v_{NS}(t)\right|\right|^{2p}_{\mathbb{L}^{2}} + 2p\nu\int_{0}^{T}\left|\left|v_{NS}(t)\right|\right|^{2(p-1)}_{\mathbb{L}^{2}}\left|\left|\nabla v_{NS}(t)\right|\right|^{2}_{\mathbb{L}^{2}}dt\right] \leq C_{2}$,
			\item $\displaystyle \mathbb{E}\left[\sup_{t \in [0,T]}\left|\left|\nabla v_{NS}(t)\right|\right|_{\mathbb{L}^{2}}^{2p} + \left(\nu\int_{0}^{T}\left|\left|Av_{NS}(t)\right|\right|^{2}_{\mathbb{L}^{2}}\right)^{p}\right] \leq C_{4}$,
		\end{enumerate}
		where $C_{2} > 0$ depends on constants $\mathcal{C}_{max}$, $C_{1}$ of Lemma~\ref{appen lemma a priori estimates 1} and its parameters, and $C_{4} > 0$ depends on $C_{1}$, $||\bar{u}_{0}||_{L^{6p}(\Omega; \mathbb{V})}$ and $\mathcal{C}_{max}$.
	\end{thm}
	
	\begin{rmk}
		Throughout this chapter, there will only be a single limit concept parameterized by $N$; no successive double limits are intended within this context. In a more accurate way, we will neither treat the case $\alpha \to 0$ while fixing $N$ nor the independent convergences of $\alpha$ and $N$. The whole study revolves around the convergence of $N$ to $+\infty$, which leads $\alpha$ to vanish.
	\end{rmk}
	
	\section{Faedo-Galerkin approximation and a priori estimates}\label{section Faedo-Galerkin}
	It is well-known (c.f. \cite[Lemma 3.1]{temam1995navier}) that the trilinear term of the Navier-Stokes equations $\displaystyle \int_{D}[z\cdot\nabla]z\Delta zdx$ vanishes if the configurations were set to two-dimensional domain with periodic boundary conditions. This property is unfortunately inapplicable to $\tilde{b}(z, z - \alpha^{2}\Delta z, \Delta z)$. Therefore, we must find a way to achieve high spatial regularity estimates. To this purpose, let $N \in \mathbb{N}\backslash\{0\}$ be a large integer, $\{e_{k}, k \geq 1\}$ be a complete orthonormal basis of $\mathbb{H}$ consisting of eigenfunctions of the Stokes operator $A$ whose domain is $\mathbb{H}^{2}\cap\mathbb{V}$, and $\left\{\mu_{k}, k \geq 1\right\}$ be the associated eigenvalues. Denote by $V_{N} \coloneqq span\{e_{1}, \dotsc, e_{N}\}$ the finite-dimensional vector subspace of $\mathbb{H}$, and by $P_{N} \colon \mathbb{H}\to \mathbb{H}$ the projection operator of $H$ onto $V_{N}$ such that for all $v \in \mathbb{H}$, it holds that
	\begin{equation*}
		\begin{aligned}
			&\left(v, \pi\right) = \left(P_{N}v, \pi\right), \ \ \forall \pi \in V_{N} , \mbox{ and } 
			\\&\left(\nabla v, \nabla \pi\right) = \left(\nabla P_{N}v, \nabla \pi\right), \ \ \forall \pi \in V_{N}.
		\end{aligned}
	\end{equation*} 
	We will assume from now on that $\mathcal{C}_{min}\mu_{N}^{-3/4}\leq\alpha \leq \mathcal{C}_{max}\mu_{N}^{-3/4}$, for some constants $\mathcal{C}_{min}, \mathcal{C}_{max} > 0$ independent of $N$. That way, when $N$ tends to $+\infty$, the spatial scale $\alpha$ goes to $0$, thanks to the property $\mu_{1} < \mu_{2} < \dotsc < \mu_{N} \to +\infty$ as $N \to \infty$. $N$ is opted to be significant to ensure that $1/\mu_{N} \leq 1$. Consequently, we introduce the following Faedo-Galerkin approximate system:
	\begin{equation}\label{eq Faedo-Galerkin}
		\begin{cases}
			\begin{aligned}
				&\left(v_{N}(t), e_{k}\right) + \nu\int_{0}^{t}\left(\nabla v_{N}(s), \nabla e_{k}\right)ds + \int_{0}^{t}\tilde{b}(\bar{u}_{N}(s), v_{N}(s), e_{k})ds \\&\hspace{20pt}= \left(v_{0}, e_{k}\right) + \left(\int_{0}^{t}g(s, \bar{u}_{k}(s))dW(s), e_{k}\right), \\&
				\left(v_{N}(t), e_{k}\right) = \left(\bar{u}_{N}(t), e_{k}\right) + \alpha^{2}\left(\nabla \bar{u}_{N}(t), \nabla e_{k}\right), 
			\end{aligned}
		\end{cases}
	\end{equation}
	for all $t\in [0,T]$, $k \in \{1, \dotsc, N\}$, and $\mathbb{P}$-almost surely, with initial datum $\bar{u}_{N}(0) = P_{N}\bar{u}_{0}$ i.e. $v_{N}(0) = P_{N}v_{0} = (P_{N} + \alpha^{2}P_{N}A)\bar{u}_{0}$. System~\eqref{eq Faedo-Galerkin} converges to the unique strong solution of the stochastic Navier-Stokes equations when $N$ tends to $+\infty$ in the sense of Definition~\ref{definition NS solution} (see Section~\ref{section convergence}). We list down below all concerned a priori estimates for the projected couple $\left(\bar{u}_{N}, v_{N}\right)$.
	\begin{rmk}
		Assumption $\alpha \leq \mathcal{C}_{max}\mu_{N}^{-3/4}$ could have been $\alpha \leq \mathcal{C}_{max}\mu_{N}^{-1/2}$ if only the convergence of solutions to equations~\eqref{main equation} toward solutions to problem~\eqref{eq NS} was intended. The additional negative exponent on $\mu_{N}$ is solely required in this context to obtain high spacial regularity for the velocities $v$ and $\bar{u}$.
	\end{rmk}
	
	\begin{lem}\label{appen lemma a priori estimates 1}
		Let $T > 0$, $N \in \mathbb{N}\backslash \{0\}$, $p \geq 1$, and assumptions \ref{S1}-\ref{S2} be valid. Then, the finite-dimensional system~\eqref{eq Faedo-Galerkin} has a $\mathbb{V}\times\mathbb{H}$-valued solution $(\bar{u}_{N}, v_{N})$ that satisfies the following estimates:
		\begin{enumerate}[label=(\roman*)]
			\item $\displaystyle \sup\limits_{0 \leq t \leq T}\mathbb{E}\left[\left|\left|\bar{u}_{N}(t)\right|\right|^{2p}_{\alpha}\right] + 2p\nu\mathbb{E}\left[\int_{0}^{T}\left|\left|\bar{u}_{N}(t)\right|\right|_{\alpha}^{2(p-1)}\left|\left|\nabla \bar{u}_{N}(t)\right|\right|^{2}_{\alpha}dt\right] \leq C_{1}$,
			\item $\displaystyle \mathbb{E}\left[\sup\limits_{0 \leq t \leq T}\left|\left|\bar{u}_{N}(t)\right|\right|^{2p}_{\alpha}\right] \leq C_{1}$,
		\end{enumerate}
		for a certain constant $C_{1} > 0$ depending only on $\mathbb{E}\left[\left|\left|\bar{u}_{0}\right|\right|^{2p}_{\mathbb{H}^{1}}\right], p, D, K_{1}, K_{2}$, and $T$. Moreover, if one assumes $\alpha \leq \mu_{N}^{-1/2}$ then, it holds that
		\begin{enumerate}
			\item[(iii)] $\displaystyle \mathbb{E}\left[\sup\limits_{0 \leq t \leq T}\left|\left|v_{N}(t)\right|\right|^{2p}_{\mathbb{L}^{2}}\right] + 2p\nu\mathbb{E}\left[\int_{0}^{T}\left|\left|v_{N}(t)\right|\right|_{\mathbb{L}^{2}}^{2(p-1)}\left|\left|\nabla v_{N}(t)\right|\right|^{2}_{\mathbb{L}^{2}}dt\right] \leq C_{2}$,
		\end{enumerate}
		where $C_{2}$ is a positive constant depending only on $C_{1}$.
	\end{lem}
	\begin{proof}
		Problem~\eqref{eq Faedo-Galerkin} is a finite-dimensional system of ordinary differential equations subject to a polynomial nonlinearity. Therefore, it has a local solution $(\bar{u}_{N}, v_{N})$. In order to apply the It{\^o} formula, we need to define, for $n \in \mathbb{N}\backslash \{0\}$, the following stopping time: 
		\begin{equation*}
			\tau_{N}^{n} \coloneqq 
			\begin{cases}
				\inf\left\{t \in [0,T]: \left|\left|(I + \alpha^{2}A)^{-1/2}v_{N}(t)\right|\right|_{\mathbb{L}^{2}} > n\right\} &\mbox{ if the set is non-empty,}\\
				+\infty &\mbox{ otherwise.}
			\end{cases}
		\end{equation*}
		For $p \geq 1$, and $t \in [0,T]$, we define the process $F(v_{N}(t)) \coloneqq \left|\left|(I + \alpha^{2}A)^{-1/2}v_{N}(t)\right|\right|^{2p}_{\mathbb{L}^{2}}$. From equation~\eqref{eq Faedo-Galerkin}$_{2}$, and taking into account that $I + \alpha^{2}A$ is self-adjoint and bijective from $D(A)$ to $\mathbb{H}$, it is straightforward that $F(v_{N}) = ||\bar{u}_{N}||_{\alpha}^{2p}$. Moreover, 
		\begin{equation*}
			\begin{aligned}
				&DF(v_{N}) = 2p||(I + \alpha^{2}A)^{-1/2}v_{N}||^{2(p-1)}_{\mathbb{L}^{2}}(I + \alpha^{2}A)^{-1}v_{N} = 2p||\bar{u}_{N}||^{2(p-1)}_{\alpha}\bar{u}_{N}, \mbox{ and }
				\\&D^{2}F(v_{N}) = 4p(p-1)||\bar{u}_{N}||^{2p-4}_{\alpha}\bar{u}_{N}\otimes\bar{u}_{N} + 2p||\bar{u}_{N}||^{2p-2}_{\alpha}(I + \alpha^{2}A)^{-1},
			\end{aligned}
		\end{equation*}
		where the symbol $\otimes$ denotes the usual dyadic product. Apply now the It{\^o} formula to the process $F(v_{N}(t \wedge \tau_{N}^{n}))$:
		\begin{equation*}
			\begin{aligned}
				&||\bar{u}_{N}(t\wedge \tau_{N}^{n})||^{2p}_{\alpha} = ||\bar{u}_{N}(0)||^{2p}_{\alpha} + 2p\int_{0}^{t\wedge\tau_{N}^{n}}||\bar{u}_{N}(s)||^{2(p-1)}_{\alpha}\left(\bar{u}_{N}(s), g(s, \bar{u}_{N}(s))dW(s)\right) \\&+ 2p(p-1)\int^{t\wedge\tau_{N}^{n}}_{0}||\bar{u}_{N}(s)||^{2p-4}_{\alpha}\left|\left|(\bar{u}_{N}(s))^{*}g(s, \bar{u}_{N}(s))\right|\right|^{2}_{K}ds \\&+ p\int_{0}^{t\wedge\tau_{N}^{n}}||\bar{u}_{N}(s)||^{2(p-1)}_{\alpha}\left|\left|(I + \alpha^{2}A)^{-1/2}g(s, \bar{u}_{N}(s))\right|\right|^{2}_{\mathscr{L}_{2}(K, \mathbb{L}^{2})}ds \\&+ 2p\int_{0}^{t\wedge\tau_{N}^{n}}||\bar{u}_{N}(s)||^{2(p-1)}_{\alpha}\langle \bar{u}_{N}(s), -\nu Av_{N}(s) - \tilde{B}(\bar{u}_{N}(s), v_{N}(s)) \rangle ds.
			\end{aligned}
		\end{equation*}
		We have $\langle \bar{u}_{N}(s), Av_{N}(s) \rangle = \langle \nabla \bar{u}_{N}(s), \nabla(I + \alpha^{2}A)\bar{u}_{N}(s) \rangle = ||\nabla\bar{u}_{N}(s)||^{2}_{\alpha}$, and by Proposition~\ref{prop trilinear term}-\textit{(i)}, the nonlinear term $\tilde{B}$ in the last term on the right-hand side of the above equation vanishes so that
		\begin{equation}\label{eq calc1}
			\begin{aligned}
				&||\bar{u}_{N}(t\wedge \tau_{N}^{n})||^{2p}_{\alpha} + 2p\nu\int_{0}^{t\wedge\tau_{N}^{n}}||\bar{u}_{N}(s)||_{\alpha}^{2p-2}||\nabla \bar{u}_{N}(s)||^{2}_{\alpha}ds 
				\\&\leq ||\bar{u}_{N}(0)||^{2p}_{\alpha} + 2p\int_{0}^{t\wedge\tau_{N}^{n}}||\bar{u}_{N}(s)||^{2p-2}_{\alpha}\left(\bar{u}_{N}(s), g(s, \bar{u}_{N}(s))dW(s)\right) 
				\\&+ 2p(p-1)\int_{0}^{t\wedge\tau_{N}^{n}}||\bar{u}_{N}(s)||^{2p-4}_{\alpha}||\bar{u}_{N}(s)||^{2}_{\mathbb{L}^{2}}||g(s, \bar{u}_{N}(s))||^{2}_{\mathscr{L}_{2}(K, \mathbb{L}^{2})}ds 
				\\&+ p\int_{0}^{t\wedge\tau_{N}^{n}}||\bar{u}_{N}(s)||^{2p-2}_{\alpha}\left|\left|(I + \alpha^{2}A)^{-1/2}g(s, \bar{u}_{N}(s))\right|\right|^{2}_{\mathscr{L}_{2}(K, \mathbb{L}^{2})}ds 
				\\&= ||\bar{u}_{N}(0)||^{2}_{\alpha} + I_{1} + I_{2} + I_{3}.
			\end{aligned}
		\end{equation}
		Assumption~\ref{S2} together with the stopping time $\tau_{N}^{n}$ yield $\mathbb{E}[I_{1}] = 0$. On the other hand, by virtue of Proposition~\ref{prop inequalities}-\textit{(i)}, assumption~\ref{S2}, and estimate $||(I + \alpha^{2}A)^{-1/2}z||_{\mathbb{L}^{2}} \leq ||z||_{\mathbb{L}^{2}}$, it holds that
		\begin{equation*}
			\begin{aligned}
				&I_{2} + I_{3} \leq 2p(p-1)\int_{0}^{t\wedge \tau_{N}^{n}}||\bar{u}_{N}(s)||^{2p-4}_{\alpha}||\bar{u}_{N}(s)||^{2}_{\mathbb{L}^{2}}\left(K_{1} + K_{2}||\bar{u}_{N}(s)||_{\alpha}\right)^{2}ds \\&+ p\int_{0}^{t\wedge \tau_{N}^{n}}||\bar{u}_{N}(s)||^{2p-2}_{\alpha}\left(K_{1} + K_{2}||\bar{u}_{N}(s)||_{\alpha}\right)^{2}ds 
				\\&\leq p(2p-1)\int_{0}^{t\wedge \tau_{N}^{n}}||\bar{u}_{N}(s)||^{2p-2}_{\alpha}\left(K_{1} + K_{2}||\bar{u}_{N}(s)||_{\alpha}\right)^{2}ds
				\\&\leq 2p(2p-1)K^{2}_{1}t\wedge\tau_{N}^{n} + 2p(2p-1)(K_{1}^{2} + K^{2}_{2})\int_{0}^{t\wedge\tau_{N}^{n}}||\bar{u}_{N}(s)||^{2p}_{\alpha}ds.
			\end{aligned}
		\end{equation*}
		Putting it all together and applying the mathematical expectation to equation~\eqref{eq calc1} return
		\begin{equation*}
			\begin{aligned}
				&\mathbb{E}\left[\lvert\lvert\bar{u}_{N}(t\wedge \tau_{N}^{n})\rvert\rvert^{2p}_{\alpha}\right] + 2p\nu\mathbb{E}\left[\int_{0}^{t\wedge\tau_{N}^{n}}||\bar{u}_{N}(s)||_{\alpha}^{2p-2}||\nabla \bar{u}_{N}(s)||^{2}_{\alpha}ds\right] \leq \mathbb{E}\left[||\bar{u}_{N}(0)||^{2p}_{\alpha}\right] \\&+ 2p(2p-1)K_{1}^{2}\mathbb{E}\left[t\wedge\tau_{N}^{n}\right] + 2p(2p-1)(K_{1}^{2} + K_{2}^{2})\int_{0}^{t\wedge\tau_{N}^{n}}\mathbb{E}\left[||\bar{u}_{N}(s)||^{2p}_{\alpha}\right]ds.
			\end{aligned}
		\end{equation*}
		The Gr{\"o}nwall inequality (c.f. \cite{ames1997inequalities}) finally implies
		\begin{equation}\label{eq calc2}
			\begin{aligned}
				&\sup\limits_{0 \leq t \leq T}\mathbb{E}\left[||\bar{u}_{N}(t\wedge \tau_{N}^{n})||^{2p}_{\alpha}\right] + 2p\nu\mathbb{E}\left[\int_{0}^{t\wedge\tau_{N}^{n}}||\bar{u}_{N}(s)||_{\alpha}^{2p-2}||\nabla \bar{u}_{N}(s)||^{2}_{\alpha}ds\right] \\&\leq \left(\mathbb{E}\left[||\bar{u}_{N}(0)||^{2p}_{\alpha}\right] + 2p(2p-1)K_{1}^{2}\mathbb{E}\left[t\wedge\tau_{N}^{n}\right]\right)\exp\left(2p(2p-1)(K_{1}^{2} + K_{2}^{2})t\wedge\tau_{N}^{n}\right).
			\end{aligned}
		\end{equation}
		Taking into account that $\mathbb{E}\left[||\bar{u}_{N}(0)||^{2p}_{\alpha}\right] \leq \mathbb{E}\left[||\bar{u}_{0}||^{2p}_{\alpha}\right]$, and letting $n \to + \infty$ in equation~\eqref{eq calc2} complete the proof of estimate \textit{(i)}. Now that we have illustrated that $||\bar{u}_{N}||_{\alpha}$ has finite moments, we can drop the stopping time in equation~\eqref{eq calc1}. whose supremum in time returns
		\begin{equation}
			\begin{aligned}\label{eq calc3'}
				&\mathbb{E}\left[\sup\limits_{0 \leq t \leq T}||\bar{u}_{N}(t)||^{2p}_{\alpha}\right] \leq \mathbb{E}\left[||\bar{u}_{N}(0)||^{2p}_{\alpha}\right] \\&+ 2p\mathbb{E}\left[\sup\limits_{0 \leq t \leq T}\left|\int_{0}^{t}||\bar{u}_{N}(s)||^{2p-2}_{\alpha}\left(\bar{u}_{N}(s), g(s, \bar{u}_{N}(s))dW(s)\right)\right|\right] \\&+ 2p(2p-1)K_{1}^{2}T + 2p(2p-1)(K_{1}^{2} + K_{2}^{2})T\sup\limits_{0 \leq t \leq T}\mathbb{E}\left[||\bar{u}_{N}(t)||^{2p}_{\alpha}\right].
			\end{aligned}
		\end{equation}
		By virtue of Proposition~\ref{prop inequalities}-\textit{(i)}, assumption~\ref{S2}, the Burkholder-Davis-Gundy (c.f. \cite{da2014stochastic}) and Young inequalities, the second term on the right-hand side can be bounded by
		\begin{equation*}
			\begin{aligned}
				&\lesssim \mathbb{E}\left[\left(\int_{0}^{T}||\bar{u}_{N}(t)||^{4p-2}_{\alpha}||g(t, \bar{u}_{N}(t))||^{2}_{\mathscr{L}_{2}(K,\mathbb{L}^{2})}dt\right)^{1/2}\right] \\&\lesssim \mathbb{E}\left[\sup\limits_{0 \leq t \leq T}||\bar{u}_{N}(t)||^{\frac{2p-1}{2}}_{\alpha}||g(t, \bar{u}_{N}(t))||^{1/2}_{\mathscr{L}_{2}(K, \mathbb{L}^{2})}\left(\int_{0}^{T}||\bar{u}_{N}(t)||_{\alpha}^{2p-1}||g(t, \bar{u}_{N}(t))||_{\mathscr{L}_{2}(K, \mathbb{L}^{2})}dt\right)^{1/2}\right] \\&\leq \frac{\varepsilon}{2}\mathbb{E}\left[K_{1} + (K_{1}  + K_{2})\sup\limits_{0 \leq t \leq T}||\bar{u}_{n}(t)||^{2p}_{\alpha}\right] + \frac{1}{2\varepsilon}\left(K_{1}T + (K_{1} + K_{2})T\sup\limits_{0 \leq t \leq T}\mathbb{E}\left[||\bar{u}_{N}(t)||^{2p}_{\alpha}\right]\right),
			\end{aligned}
		\end{equation*}
		for some constant $\varepsilon>0$ emerging from the Young inequality. Taking $\varepsilon = \frac{1}{K_{1} + K_{2}}$, merging the above result into equation~\eqref{eq calc3'}, and employing assertion \textit{(i)} complete the proof of estimate \textit{(ii)}. Moving on to the inequality \textit{(iii)}, we have $\left(v_{N}, \psi\right) = \left(\bar{u}_{N}, \psi\right) + \alpha^{2}\left(\nabla\bar{u}_{N}, \nabla \psi\right)$ $\mathbb{P}$-a.s. for all $\psi \in V_{N}$, thanks to equation~\eqref{eq Faedo-Galerkin}$_{2}$. Therefore, substituting $\psi$ by $v_{N}(t)$ and employing the Cauchy-Schwarz inequality to get: $||v_{N}(t)||^{2}_{\mathbb{L}^{2}} \leq ||\bar{u}_{N}(t)||_{\mathbb{L}^{2}}||v_{N}(t)||_{\mathbb{L}^{2}} + \alpha^{2}||\nabla\bar{u}_{N}(t)||_{\mathbb{L}^{2}}||\nabla v_{N}(t)||_{\mathbb{L}^{2}}$. On the other hand, the estimate $||\nabla v_{N}(t)||_{\mathbb{L}^{2}} \leq \sqrt{\mu_{N}}||v_{N}(t)||_{\mathbb{L}^{2}}$ together with the hypothesis $\alpha \leq \mu_{N}^{-1/2}$ and the Young inequality lead to
		\begin{equation}\label{eq calc4}
			||v_{N}(t)||_{\mathbb{L}^{2}} \leq \sqrt{2}||\bar{u}_{N}(t)||_{\alpha}.
		\end{equation}
		Following the same technique, but this time replacing $\psi$ by $Av_{N}(t) \in V_{N}$, one obtains
		\begin{equation}\label{eq calc5}
			||\nabla v_{N}(t)||_{\mathbb{L}^{2}} \leq \sqrt{2}||\nabla \bar{u}_{N}(t)||_{\alpha}.
		\end{equation}
		It suffices now to raise inequality~\eqref{eq calc4} to the power $2p$, take the supremum over $t \in [0, T]$, apply to it the mathematical expectation, and employ estimate \textit{(ii)} to get $\mathbb{E}\left[\sup\limits_{0 \leq t \leq T}||v_{N}(t)||^{2p}_{\mathbb{L}^{2}}\right] \lesssim C_{1}$. Similarly, $||v_{N}(t)||^{2(p-1)}_{\mathbb{L}^{2}}||\nabla v_{N}(t)||^{2}_{\mathbb{L}^{2}} \lesssim ||\bar{u}_{N}(t)||^{2(p-1)}_{\alpha}||\nabla \bar{u}_{N}(t)||^{2}_{\alpha}$, thanks to \eqref{eq calc4} and \eqref{eq calc5}. Integrating over $[0,T]$, applying the mathematical expectation and employing estimate \textit{(i)} terminate the proof.
	\end{proof}
	
	The next lemma exhibits the regularity of $v_{0}$ with respect to $\bar{u}_{0}$.
	\begin{lem}\label{lemma v0 regularity}
		Let $1 \leq p < +\infty$, and assume \ref{S1}. If $\mathcal{C}_{min}\mu_{N}^{-1/2} \leq \alpha \leq \mathcal{C}_{max}\mu_{N}^{-1/2}$ for some constants $\mathcal{C}_{min}, \mathcal{C}_{max} > 0$, then $v_{0} \in L^{2p}(\Omega; \mathbb{V})$, and 
		\begin{equation*}
			\left|\left|\nabla v_{0}\right|\right|_{L^{2p}(\Omega; \mathbb{L}^{2})} \leq \left|\left|\bar{u}_{0}\right|\right|_{L^{2p}(\Omega; \mathbb{V})}.
		\end{equation*}
	\end{lem}
	\begin{proof}
		By equation~\eqref{eq Faedo-Galerkin}$_{2}$, we get 
		\begin{equation*}
			||\nabla v_{N}(0)||_{\mathbb{L}^{2}}^{2} = \left(\nabla\bar{u}_{N}(0), \nabla v_{N}(0)\right) + \alpha^{2}\left(A\bar{u}_{N}(0), A v_{N}(0)\right).
		\end{equation*}
		Taking into account the estimate $||Az||_{\mathbb{L}^{2}} \leq \sqrt{\mu_{N}}||\nabla z||_{\mathbb{L}^{2}}$ for all $z \in V_{N}$, apply it to $||A\bar{u}_{N}(0)||_{\mathbb{L}^{2}}$ and $||Av_{N}(0)||_{\mathbb{L}^{2}}$, and employ the Cauchy-Schwarz inequality, it follows $||\nabla v_{N}(0)||_{\mathbb{L}^{2}} \leq 2||\nabla \bar{u}_{N}(0)||_{\mathbb{L}^{2}}$. Subsequently, $\mathbb{E}\left[||\nabla v_{N}(0)||^{2p}_{\mathbb{L}^{2}}\right] \lesssim \mathbb{E}\left[||\nabla\bar{u}_{0}||^{2p}_{\mathbb{L}^{2}}\right] \eqqcolon M$, which implies that $\left(v_{N}(0)\right)_{N}$ is bounded in the reflexive Banach space $L^{2p}(\Omega; \mathbb{H}^{1})$. Thus, there exists a subsequence $\left(v_{N_{\ell}}(0)\right)_{\ell}$ that converges weakly in $L^{2p}(\Omega; \mathbb{H}^{1})$ toward some limit $\xi$, and one gets $\mathbb{E}\left[||\xi||^{2p}_{\mathbb{H}^{1}}\right] \leq \liminf \mathbb{E}\left[||v_{N_{\ell}}(0)||^{2p}_{\mathbb{H}^{1}}\right] \leq C_{D}M$, thanks to the Poincar{\'e} inequality. It remains to identify $\xi$ with $v_{0}$. Indeed, since $L^{2p}(\Omega; \mathbb{H}^{1}) \hookrightarrow L^{2p}(\Omega; \mathbb{L}^{2})$, the weak convergence of $\left(v_{N_{\ell}}(0)\right)_{\ell}$ also takes place in $L^{2p}(\Omega; \mathbb{L}^{2})$. Observe that $v_{N}(0) = P_{N}v_{0}$ converges strongly (and therefore weakly) toward $v_{0}$ in $L^{2p}(\Omega; \mathbb{L}^{2})$ as $N \to +\infty$, thanks to the properties of the projector $P_{N}$. Consequently, by the weak limit uniqueness, $\xi = v_{0}$ $\mathbb{P}$-a.s. and a.e. in $D$, and the result follows.
	\end{proof}
	
	Owing to Lemma~\ref{lemma v0 regularity}, high space-regularity estimates are illustrated below for the process $\left(\bar{u}_{N}, v_{N}\right)$.
	\begin{lem}\label{appen lemma a priori estimates 2}
		Let $N \in \mathbb{N}\backslash \{0\}$, and $p \in [1, +\infty)$. Assume that \ref{S1}-\ref{S2} are valid and that $\alpha \leq \mathcal{C}_{max}\mu_{N}^{-3/4}$, for some constant $\mathcal{C}_{max} > 0$ independent of $N$. Then, the solution $\left(\bar{u}_{N}, v_{N}\right)$ of equation~\eqref{eq Faedo-Galerkin} satisfies
		\begin{enumerate}[label=(\roman*)]
			\item $\displaystyle \mathbb{E}\left[\sup_{t \in [0,T]}\left|\left|\nabla \bar{u}_{N}(t)\right|\right|^{2p}_{\mathbb{\alpha}} + \left(\nu\int_{0}^{T}\left|\left|A\bar{u}_{N}(t)\right|\right|^{2}_{\alpha}dt\right)^{p}\right] \leq C_{3}$,
			\item $\displaystyle \mathbb{E}\left[\sup_{t \in [0,T]}\left|\left|\nabla v_{N}(t)\right|\right|^{2p}_{\mathbb{L}^{2}} + \left(\nu\int_{0}^{T}\left|\left|Av_{N}(t)\right|\right|^{2}_{\mathbb{L}^{2}}dt\right)^{p}\right] \leq C_{4}$,
		\end{enumerate}
		where $C_{3} > 0$ depends on $C_{1}$ and $\left|\left|\bar{u}_{0}\right|\right|_{L^{6p}(\Omega; \mathbb{V})}$, and $C_{4}$ depends only on $C_{3}$ and $\mathcal{C}_{max}$.
	\end{lem}
	\begin{proof}
		Define the stopping time
		\begin{equation*}
			\tau_{N}^{n} \coloneqq 
			\begin{cases}
				\inf\left\{t \in [0,T]: \left|\left|A^{1/2}(I + \alpha^{2}A)^{-1/2}v_{N}(t)\right|\right|_{\mathbb{L}^{2}} > n\right\} &\mbox{ if the set is non-empty,}\\
				+\infty &\mbox{ otherwise,}
			\end{cases}
		\end{equation*}
		and the process $F(v_{N}) \coloneqq ||A^{1/2}(I + \alpha^{2}A)^{-1/2}v_{N}||^{2}_{\mathbb{L}^{2}}$. By equation~\eqref{eq Faedo-Galerkin}$_{2}$, one gets $F(v_{N}(t \wedge \tau_{N}^{n})) = ||\nabla \bar{u}_{N}(t\wedge \tau_{N}^{n})||^{2}_{\alpha}$. Moreover, $DF(x) = 2A(I + \alpha^{2}A)^{-1}x$, and  $D^{2}F(x) = 2A(I + \alpha^{2}A)^{-1}$. In particular, $DF(v_{N}) = 2A\bar{u}_{N}$, thanks to equation~\eqref{eq Faedo-Galerkin}$_{2}$.By applying It{\^o}'s formula to the process $F(v_{N}(t \wedge \tau_{N}^{n}))$, it follows that
		\begin{equation*}
			\begin{aligned}
				&||\nabla \bar{u}_{N}(t\wedge \tau_{N}^{n})||^{2}_{\alpha} + 2\nu\int_{0}^{t\wedge \tau_{N}^{n}}||A\bar{u}_{N}(s)||^{2}_{\alpha}ds = ||\nabla \bar{u}_{N}(0)||^{2}_{\alpha} 
				\\&+ 2\int_{0}^{t\wedge \tau_{N}^{n}}\left(A\bar{u}_{N}(s), g(s, \bar{u}_{N}(s))dW(s)\right) + \int_{0}^{t\wedge \tau_{N}^{n}}||A^{1/2}(I + \alpha^{2}A)^{-1/2}g(s, \bar{u}_{N}(s))||^{2}_{\mathscr{L}_{2}(K, \mathbb{L}^{2})}ds 
				\\&- 2\int_{0}^{t\wedge \tau_{N}^{n}}\langle \tilde{B}(\bar{u}_{N}(s), v_{N}(s)), A\bar{u}_{N}(s) \rangle ds = ||\nabla \bar{u}_{N}(0)||^{2}_{\alpha} + I_{1} + I_{2} - I_{3}.
			\end{aligned}
		\end{equation*}
		On account of assumption \ref{S2} and the measurability of $\bar{u}_{N}$, we have $\mathbb{E}\left[I_{1}\right] = 0$. Now, the fact that for any $z \in \mathbb{L}^{2}$, the quantity $||A^{1/2}(I + \alpha^{2}A)^{-1/2}z||_{\mathbb{L}^{2}}$ is optimally bounded by $\frac{1}{\alpha}||z||_{\mathbb{L}^{2}}$ justifies the opted assumption on $||g(\cdot, z)||_{\mathscr{L}_{2}(K, \mathbb{H}^{1})}$. Therewith, \ref{S2} leads to
		\begin{equation*}
			\begin{aligned}
				I_{2} \leq \int_{0}^{t\wedge\tau_{N}^{n}}||g(s, \bar{u}_{N}(s))||^{2}_{\mathscr{L}_{2}(K, \mathbb{H}^{1})}ds \leq 2K_{1}^{2}t\wedge\tau_{N}^{n} + 2K_{2}^{2}\int_{0}^{t\wedge\tau_{N}^{n}}||\bar{u}_{N}(s)||^{2}_{\alpha}ds.
			\end{aligned}
		\end{equation*}
		Moreover, by virtue of equation~\eqref{eq Faedo-Galerkin}$_{2}$, the identity $v_{N} = \bar{u}_{N} + \alpha^{2}A\bar{u}_{N}$ holds $\mathbb{P}$-a.s. and a.e. in $(0,T)\times D$. Thus, the integrand of $I_{3}$ can be amended to the following form
		\begin{equation*}
			\begin{aligned}
				\left( \tilde{B}(\bar{u}_{N}(s), \bar{u}_{N}(s)), A\bar{u}_{N}(s)\right) + \alpha^{2}\left( \tilde{B}(\bar{u}_{N}(s), A\bar{u}_{N}(s)), A\bar{u}_{N}(s) \right) \eqqcolon B_{1} + B_{2}.
			\end{aligned}
		\end{equation*}
		Proposition~\ref{prop trilinear term}-\textit{(ii)} yields $B_{1} = \left([\bar{u}_{N}(s)\cdot\nabla]\bar{u}_{N}(s), A\bar{u}_{N}(s)\right) - \left([A\bar{u}_{N}(s)\cdot\nabla]\bar{u}_{N}(s), \bar{u}_{N}(s)\right)$. The first term vanishes thanks to \cite[Lemma 3.1]{temam1995navier}, as well as the second term (see Proposition~\ref{prop trilinear term NS}-\textit{(i)}). Thereby, $B_{1} = 0$. On the other hand, by Proposition~\ref{prop trilinear term}-\textit{(ii)}, it follows $B_{2} = - \alpha^{2}\left([A\bar{u}_{N}(s)\cdot\nabla]A\bar{u}_{N}(s), \bar{u}_{N}(s)\right)$. Hence,
		\begin{equation*}
			\begin{aligned}
				&\left|I_{3}\right| \leq 2\alpha^{2}C_{D}\int_{0}^{t\wedge \tau_{N}^{n}}||A\bar{u}_{N}(s)||_{\mathbb{L}^{2}}||A^{3/2}\bar{u}_{N}(s)||_{\mathbb{L}^{2}}||\bar{u}_{N}(s)||^{1/2}_{\mathbb{L}^{2}}||A\bar{u}_{N}(s)||^{1/2}_{\mathbb{L}^{2}}ds \\&\leq 2\alpha^{2}C_{D}\mu_{N}^{3/2}\int_{0}^{t\wedge\tau_{N}^{n}}||A\bar{u}_{N}(s)||_{\mathbb{L}^{2}}^{3/2}||\bar{u}_{N}(s)||^{3/2}_{\mathbb{L}^{2}}ds \\&\leq \frac{4\mathcal{C}_{max}^{8}C^{4}_{D}}{\nu^{3}}\int_{0}^{t\wedge\tau_{N}^{n}}||\bar{u}_{N}(s)||^{6}_{\mathbb{L}^{2}}ds + \frac{3\nu}{4}\int_{0}^{t\wedge\tau_{N}^{n}}||A\bar{u}_{N}(s)||^{2}_{\mathbb{L}^{2}}ds.
			\end{aligned}
		\end{equation*}
		where Proposition~\ref{prop trilinear term NS}-\textit{(ii)}, estimate $||A^{3/2}z||_{\mathbb{L}^{2}}\leq \mu_{N}^{3/2}||z||_{\mathbb{L}^{2}}$, for all $z\in V_{N}$, condition $\alpha \leq \mathcal{C}_{max}\mu_{N}^{-3/4}$ together with the Young inequality with conjugate exponents $1/4$ and $3/4$ were taken advantage of. Observe that 
		\begin{equation*}
			\begin{aligned}
				&||\nabla\bar{u}_{N}(0)||_{\alpha}^{2} = ||\nabla \bar{u}_{N}(0)||^{2}_{\mathbb{L}^{2}} + \alpha^{2}||A\bar{u}_{N}(0)||^{2}_{\mathbb{L}^{2}} \leq ||\nabla \bar{u}_{N}(0)||^{2}_{\mathbb{L}^{2}} + \mathcal{C}_{max}^{2}||\nabla\bar{u}_{N}(0)||^{2}_{\mathbb{L}^{2}} 
				\\&\leq (1 + \mathcal{C}_{max}^{2})||\nabla \bar{u}_{0}||^{2}_{\mathbb{L}^{2}},
			\end{aligned}
		\end{equation*}
		thanks to $\alpha \leq \mathcal{C}_{max}/\mu_{N}^{3/4} \leq \mathcal{C}_{max}$, and estimate $||Az||_{\mathbb{L}^{2}} \leq \sqrt{\mu_{N}}||\nabla z||_{\mathbb{L}^{2}}$ for all $z \in V_{N}$.	
		Taking into account that $\sup\limits_{0 \leq t \leq T}||\bar{u}_{N}(t)||^{2q}_{\alpha}$ is almost surely finite for all $q \geq 2$ on account of Lemma~\ref{appen lemma a priori estimates 1}, the stopping time of last and first terms on the right-hand side of $I_{2}$ and $I_{3}$ can be omitted. Thereby,
		\begin{equation}\label{eq calc15}
			\begin{aligned}
				&||\nabla\bar{u}_{N}(t\wedge\tau_{N}^{n})||^{2}_{\alpha} + \frac{5\nu}{4}\int_{0}^{t\wedge\tau_{N}^{n}}||A\bar{u}_{N}(s)||^{2}_{\alpha}ds \leq (1 + \mathcal{C}_{max}^{2})||\nabla \bar{u}_{0}||^{2}_{\mathbb{L}^{2}} + 2K_{1}^{2}t\wedge\tau_{N}^{n}
				\\&+ 2\int_{0}^{t\wedge\tau_{N}^{n}}\left(A\bar{u}_{N}(s), g(s, \bar{u}_{N}(s))dW(s)\right) + 2K_{2}^{2}\int_{0}^{t}||\bar{u}_{N}(s)||^{2}_{\alpha}ds 
				\\&+ \frac{4\mathcal{C}_{max}^{8}C_{D}^{4}}{\nu^{3}}\int_{0}^{t}||\bar{u}_{N}(s)||^{6}_{\mathbb{L}^{2}}ds.
			\end{aligned}
		\end{equation}
		Subsequently, taking the mathematical expectation, employing Lemma~\ref{appen lemma a priori estimates 1} and letting $n \to +\infty$ imply
		\begin{equation}\label{eq calc16}
			\begin{aligned}
				&\mathbb{E}\left[||\nabla \bar{u}_{N}(t)||^{2}_{\alpha}\right] + \frac{5\nu}{4}\mathbb{E}\left[\int_{0}^{t}||A\bar{u}_{N}(s)||^{2}_{\alpha}ds\right] \\&\leq (1 + \mathcal{C}_{max}^{2})\mathbb{E}\left[||\nabla\bar{u}_{0}||^{2}_{\mathbb{L}^{2}}\right] + 2K_{1}^{2}T + (2K_{2}^{2} + \frac{4\mathcal{C}_{max}^{8}C_{D}^{4}}{\nu^{3}})TC_{1}.
			\end{aligned}
		\end{equation}
		We now raise equation~\eqref{eq calc15} to the power $p$, use Proposition~\ref{prop inequalities}-\textit{(iii)}, and drop the stopping time $\tau_{N}^{n}$, thanks to estimate~\eqref{eq calc16}. We obtain
		\begin{equation}\label{eq calc17}
			\begin{aligned}
				&\sup\limits_{0 \leq t \leq T}||\nabla\bar{u}_{N}(t)||^{2p}_{\alpha} + \left(\frac{5\nu}{4}\int_{0}^{T}||A\bar{u}_{N}(t)||^{2}_{\alpha}dt\right)^{p} \lesssim ||\nabla \bar{u}_{0}||^{2p}_{\mathbb{L}^{2}} + (K_{1}^{2}T)^{p} \\&+ \left(\sup\limits_{0 \leq t \leq T}\int_{0}^{t}\left(\nabla\bar{u}_{N}(s), \nabla g(s, \bar{u}_{N}(s))dW(s)\right)\right)^{p} + (K_{2}^{2}T)^{p}\sup\limits_{0 \leq t \leq T}||\bar{u}_{N}(t)||^{2p}_{\alpha} \\&+ (\mathcal{C}_{max}^{8}C_{D}^{4}T/\nu^{3})^{p}\sup\limits_{0 \leq t \leq T}||\bar{u}_{N}(t)||^{6p}_{\mathbb{L}^{2}}.
			\end{aligned}
		\end{equation}
		We bound the third term on the right-hand side using the Burkholder-Davis-Gundy and Young inequalities, assumption~\ref{S2}, and Proposition~\ref{prop inequalities}-\textit{(iii)}:
		\begin{equation*}
			\begin{aligned}
				&\mathbb{E}\left[\left(\sup\limits_{0 \leq t \leq T}\int_{0}^{t}\left(\nabla\bar{u}_{N}(s), \nabla g(s, \bar{u}_{N}(s))dW(s)\right)\right)^{p}\right]\\&\lesssim \mathbb{E}\left[\left(\int_{0}^{T}||\nabla \bar{u}_{N}(t)||^{2}_{\mathbb{L}^{2}}||\nabla g(t, \bar{u}_{N}(t))||^{2}_{\mathscr{L}_{2}(K; \mathbb{L}^{2})}dt\right)^{p/2}\right] 
				\\&\leq \frac{1}{2}\mathbb{E}\left[\sup\limits_{0 \leq t \leq T}||\nabla \bar{u}_{N}(t)||^{2p}_{\alpha}\right] + 2^{2p-2}T^{p}\mathbb{E}\left[K_{1}^{2p} + K_{2}^{2p}\sup\limits_{0 \leq t \leq T}||\bar{u}_{N}(t)||^{2p}_{\alpha}\right],
			\end{aligned}
		\end{equation*}
		Taking afterwards the mathematical expectation of equation~\eqref{eq calc17} and employing Lemma~\ref{appen lemma a priori estimates 1} complete the proof of estimate \textit{(i)}. On the other hand, $||\nabla v_{N}(t)||^{2}_{\mathbb{L}^{2}} \leq 2\max(1, \mathcal{C}_{max}^{2})||\nabla\bar{u}_{N}(t)||^{2}_{\alpha}$ holds for all $t \in [0,T]$, thanks to equation~\eqref{eq calc5} which is slightly amended here to fit the case $\alpha \leq \mathcal{C}_{max}\mu_{N}^{-3/4}$. Furthermore, multiplying in $\mathbb{L}^{2}$ the identity $v_{N}(t) = \bar{u}_{N}(t) + \alpha^{2}A\bar{u}_{N}(t)$ by $A^{2}v_{N}$ and making use of Cauchy-Schwarz inequality give 
		\begin{equation*}
			||Av_{N}(t)||_{\mathbb{L}^{2}}^{2} \leq ||A\bar{u}_{N}(t)||_{\mathbb{L}^{2}}||Av_{N}(t)||_{\mathbb{L}^{2}} + \alpha^{2}||A^{3/2}\bar{u}_{N}(t)||_{\mathbb{L}^{2}}||A^{3/2}v_{N}(t)||_{\mathbb{L}^{2}}.
		\end{equation*}
		We use $\alpha \leq \mathcal{C}_{max}\mu_{N}^{-3/4}$, $||A^{3/2}v_{N}||_{\mathbb{L}^{2}} \leq \sqrt{\mu_{N}}||Av_{N}||_{\mathbb{L}^{2}}$, and simplify by $||Av_{N}(t)||_{\mathbb{L}^{2}}$ to obtain eventually 
		\begin{equation*}
			||Av_{N}(t)||_{\mathbb{L}^{2}} \leq \max(1, \mathcal{C}_{max})\left(||A\bar{u}_{N}(t)||_{\mathbb{L}^{2}} + \alpha||A^{3/2}\bar{u}_{N}(t)||_{\mathbb{L}^{2}}\right).
		\end{equation*}
		Squaring both sides offers $||Av_{N}(t)||_{\mathbb{L}^{2}}^{2} \leq 2\max(1, \mathcal{C}_{max}^{2})||A\bar{u}_{N}(t)||^{2}_{\alpha}$. The proof of inequality \textit{(ii)} follows after applying estimate \textit{(i)}.
	\end{proof}
	
	\section{Convergence of system~\eqref{eq Faedo-Galerkin}}\label{section convergence}
	This section is devoted to proving the convergence of $\left(\bar{u}_{N}, v_{N}\right)$ towards the unique strong solution of the stochastic Navier-Stokes equations. The followed steps are typical: we first need to bound each item of system~\eqref{eq Faedo-Galerkin} in a reflexive Banach space. Then, limits identification shall be carried out to match all Navier-Stokes problem's terms.\newline
	\noindent\textbf{Boundedness and convergence:} Now that all data are clear, we begin by bounding each term of equations~\eqref{eq Faedo-Galerkin} in a suitable reflexive Banach space. The bilinear operator $\{\tilde{B}(\bar{u}_{N}, v_{N})\}_{N}$ is bounded in $L^{2}(\Omega; L^{2}(0,T; \mathbb{V}'))$. Indeed, Proposition~\ref{prop trilinear term}-\textit{(iii)}, the embedding $\mathbb{H}^{1} \hookrightarrow \mathbb{L}^{4}$, the Cauchy-Schwarz inequality, and Lemma~\ref{appen lemma a priori estimates 2} yield
	\begin{equation*}
		\begin{aligned}
			\mathbb{E}\left[\int_{0}^{T}\left|\left|\tilde{B}(\bar{u}_{N}(t), v_{N}(t))\right|\right|^{2}_{\mathbb{V}'}dt \right] \leq C_{D}\mathbb{E}\left[\sup_{t \in [0,T]}\left|\left|\nabla\bar{u}_{N}(t)\right|\right|^{2}_{\mathbb{L}^{2}} \int_{0}^{T}\left|\left|\nabla v_{N}(t)\right|\right|^{2}_{\mathbb{L}^{2}}dt\right] \leq C_{D}C_{3}C_{4}.
		\end{aligned}
	\end{equation*}
	Therefore, setting $R(\bar{u}_{N}) \coloneqq -\nu\Delta v_{N} + \tilde{B}(\bar{u}_{N}, v_{N})$, we conclude from Lemma~\ref{appen lemma a priori estimates 1} that $\{R(\bar{u}_{N})\}_{N}$ is bounded in $L^{2}(\Omega; L^{2}(0,T; \mathbb{V}'))$.	Moreover, by virtue of Lemma~\ref{appen lemma a priori estimates 1} and assumption~\ref{S2}, $\{v_{N}\}_{N}$, $\{\bar{u}_{N}\}_{N}$ are bounded in $L^{2}(\Omega; L^{\infty}(0,T; \mathbb{H}) \cap L^{2}(0,T; \mathbb{V}))$, and $\{g(\cdot, \bar{u}_{N})\}_{N}$ too in the Hilbert space $L^{2}(\Omega; L^{2}(0,T; \mathscr{L}_{2}(K, \mathbb{L}^{2})))$. This implies the existence of of two subsequences $\{v_{N_{\ell}}\}_{\ell}$, $\{\bar{u}_{N_{\ell}}\}_{\ell}$ of $\{v_{N}\}_{N}$, $\{\bar{u}_{N}\}_{N}$ respectively, and four limiting functions $v_{NS}$, $u_{NS} \in L^{2}(\Omega; L^{\infty}(0,T; \mathbb{H}) \cap L^{2}(0,T; \mathbb{V}))$, $R_{0} \in L^{2}(\Omega; L^{2}(0, T; \mathbb{V}'))$, and $g_{0} \in L^{2}(\Omega; L^{2}(0,T; \mathscr{L}_{2}(K, \mathbb{L}^{2})))$ such that
	\begin{align}
		&v_{N_{\ell}} \rightharpoonup v_{NS} \ \ \& \ \ \bar{u}_{N_{\ell}} \rightharpoonup u_{NS} \mbox{ (weakly)} &\mbox{ in } L^{2}(\Omega; L^{2}(0,T; \mathbb{V})), \label{appen conv1} \\&
		v_{N_{\ell}} \overset{\ast}{\rightharpoonup} v_{NS} \ \ \& \ \ \bar{u}_{N_{\ell}} \overset{\ast}{\rightharpoonup} u_{NS} \mbox{ (weakly-$\ast$)} &\mbox{ in } L^{2}(\Omega; L^{\infty}(0,T; \mathbb{H})),\label{appen conv2} \\&
		R(\bar{u}_{N_{\ell}}) \rightharpoonup R_{0} \mbox{ (weakly)} &\mbox{ in } L^{2}(\Omega; L^{2}(0,T; \mathbb{V}')), \label{appen conv3} \\&
		g(\cdot, \bar{u}_{N_{\ell}}) \rightharpoonup g_{0} \mbox{ (weakly)} &\mbox{ in } L^{2}(\Omega; L^{2}(0,T; \mathscr{L}_{2}(K, \mathbb{L}^{2}))).\label{appen conv4}
	\end{align}
	As a result, the limiting function $v_{NS}$ satisfies $\mathbb{P}$-a.s. and for all $t \in [0,T]$ the equation:
	\begin{equation}\label{eq limiting equation}
		\left(v_{NS}(t), \varphi\right) + \int_{0}^{t}\langle R_{0}(s), \varphi\rangle ds = \left(v_{0}, \varphi\right) + \left(\int_{0}^{t}g_{0}(s)dW(s), \varphi\right), \ \ \forall \varphi \in \mathbb{V},
	\end{equation}
	where we recall the $v_{0}$ is the limit of $P_{N}v_{0}$ as $N \to +\infty$ in $L^{4}(\Omega; \mathbb{H})$. Making use of the classical approach in \cite{Pardoux1975}, and taking into account equation~\eqref{eq limiting equation} which is fulfilled by $v_{NS}$, it is straightforward to show that $v_{NS} \in L^{2}(\Omega; C([0,T]; \mathbb{H}))$. Besides, identity $v_{N_{\ell}} = \bar{u}_{N_{\ell}} + \alpha^{2}A\bar{u}_{N_{\ell}}$ grants equality between processes $u_{NS}$ and $v_{NS}$. Indeed, for all $\varphi \in \mathbb{H}$, it holds that 
	\begin{equation*}
		\left|\alpha^{2}\mathbb{E}\left[\int_{0}^{T}\left(A\bar{u}_{N_{\ell}}(t), \varphi\right)dt\right]\right| \leq \alpha||\varphi||_{\mathbb{L}^{2}}\mathbb{E}\left[\int_{0}^{T}\alpha^{2}||A\bar{u}_{N_{\ell}}(t)dt||^{2}_{\mathbb{L}^{2}}dt\right]^{1/2} \leq \alpha||\varphi||_{\mathbb{L}^{2}}C_{1} \to 0
	\end{equation*}
	as $\ell \to +\infty$, thanks to the hypothesis $\alpha \leq \mathcal{C}_{max}\mu_{N}^{-3/4}$. Subsequently, $\{\alpha^{2}A\bar{u}_{N_{\ell}}\}_{\ell}$ converges weakly in $L^{2}(\Omega; L^{2}(0,T; \mathbb{L}^{2}))$ to $0$, which offers, by the use of the aforementioned identity together with \eqref{appen conv1}, the equality $u_{NS} = v_{NS}$ $\mathbb{P}$-a.s. and a.e. in $[0,T] \times D$. The only remaining task in this section consists in identifying $R_{0}$ and $g_{0}$ with their solution-dependent counterparts. To this purpose, we must first state one essential property that enables such an identification.
	\begin{prop}\label{prop monotonicity}
		For $N \in \mathbb{N}\backslash \{0\}$, assume that $\alpha \leq \mathcal{C}_{max}\mu_{N}^{-3/4}$. Let $v_{N}^{1}, v_{N}^{2}$ be two vector fields in $V_{N}$ such that $v_{N}^{1} = \bar{u}_{N}^{1} + \alpha^{2}A\bar{u}_{N}^{1}$ and $v_{N}^{2} = \bar{u}_{N}^{2} + \alpha^{2}A\bar{u}_{N}^{2}$. If $L_{g} \leq \frac{\sqrt{\nu}}{C_{P}\sqrt{2}}$ then, there exists a constant $\mathcal{K} > 0$ depending only on $D$ and $\mathcal{C}_{max}$ such that
		\begin{equation*}
			\begin{aligned}
				&\left\langle -\nu\Delta (v_{N}^{1} - v_{N}^{2}) + \tilde{B}(\bar{u}_{N}^{1}, v_{N}^{1}) - \tilde{B}(\bar{u}_{N}^{2}, v_{N}^{2}) + \frac{\mathcal{K}}{\nu^{3}}\left|\left|\bar{u}_{N}^{2}\right|\right|_{\mathbb{L}^{4}}^{4}w_{N}, w_{N}\right\rangle 
				\\&- \left|\left|g(\cdot, \bar{u}_{N}^{1}) - g(\cdot, \bar{u}_{N}^{2})\right|\right|^{2}_{\mathscr{L}_{2}(K, \mathbb{L}^{2})} \geq 0,
			\end{aligned}
		\end{equation*}
		where $C_{P} > 0$ is the Poincar{\'e} constant and $w_{N} \coloneqq \bar{u}_{N}^{1} - \bar{u}_{N}^{2}$.
	\end{prop}
	\begin{proof}
		$\langle -\nu\Delta(v_{N}^{1} - v_{N}^{2}), w_{N} \rangle = \nu\left(A^{1/2}(I + \alpha^{2}A)w_{N}, A^{1/2}w_{N}\right) = \nu||\nabla w_{N}||^{2}_{\alpha}$. Besides, Proposition~\ref{prop trilinear term}-\textit{(i)} and \textit{(iii)} yield 
		\begin{equation}\label{eq calc3}
			\begin{aligned}
				&\left|\langle \tilde{B}(\bar{u}_{N}^{1}, v_{N}^{1}) - \tilde{B}(\bar{u}_{N}^{2}, v_{N}^{2}), w_{N}\rangle\right| = \left|\langle \tilde{B}(\bar{u}_{N}^{2}, v_{N}^{1} - v_{N}^{2}), w_{N}\rangle\right| \\&\leq C_{D}||\bar{u}_{N}^{2}||_{\mathbb{L}^{4}}||\nabla(v_{N}^{1} - v_{N}^{2})||_{\mathbb{L}^{2}}||w_{N}||^{\frac{1}{2}}_{\mathbb{L}^{2}}||\nabla w_{N}||^{\frac{1}{2}}_{\mathbb{L}^{2}},
			\end{aligned}
		\end{equation}
		where identity $v_{N}^{1} - v_{N}^{2} = w_{N} + \alpha^{2}Aw_{N}$ implies $\nabla (v_{N}^{1} - v_{N}^{2}) = \nabla w_{N} + \alpha^{2}\nabla Aw_{N}$ and therefore, it follows that $||\nabla(v_{N}^{1} - v_{N}^{2})||_{\mathbb{L}^{2}} \leq (1 + \mathcal{C}_{max})||\nabla w_{N}||_{\mathbb{L}^{2}}$, thanks to the condition $\alpha \leq \mathcal{C}_{max}\mu_{N}^{-3/4}$. Plugging this result back into equation~\eqref{eq calc3} and applying the Young inequality to get
		\begin{equation*}
			\left|\langle \tilde{B}(\bar{u}_{N}^{1}, v_{N}^{1}) - \tilde{B}(\bar{u}_{N}^{2}, v_{N}^{2}), w_{N}\rangle\right| \leq \frac{\nu}{4}||\nabla w_{N}||^{2}_{\mathbb{L}^{2}} + \frac{\mathcal{K}}{\nu^{3}}||\bar{u}_{N}^{2}||_{\mathbb{L}^{4}}^{4}||w_{N}||^{2}_{\mathbb{L}^{2}},
		\end{equation*}
		where $\mathcal{K} > 0$ depends only on $C_{D}$ and $\mathcal{C}_{max}$. Assumption~\ref{S2} implies $-||g(\cdot,. \bar{u}_{N}^{1}) - g(\cdot, \bar{u}_{N}^{2})||_{\mathscr{L}_{2}(K, \mathbb{L}^{2})}^{2} \geq - L_{g}^{2}||w_{N}||_{\alpha}^{2}$ in addition. Putting it all together and employing the Poincar{\'e} inequality, we obtain
		\begin{equation*}
			\begin{aligned}
				\\&\langle -\nu\Delta (v_{N}^{1} - v_{N}^{2}) + \tilde{B}(\bar{u}_{N}^{1}, v_{N}^{1}) - \tilde{B}(\bar{u}^{2}_{N}, v_{N}^{2}) + \frac{\mathcal{K}}{\nu^{3}}||\bar{u}_{N}^{2}||^{4}_{\mathbb{L}^{4}}w_{N}, w_{N} \rangle - ||g(\cdot, \bar{u}_{N}^{1}) - g(\cdot, \bar{u}_{N}^{2})||^{2}_{\mathscr{L}_{2}(K, \mathbb{L}^{2})} \\&\geq (\frac{\nu}{2} - L_{g}^{2}C_{P}^{2})||\nabla w_{N}||^{2}_{\mathbb{L}^{2}} + \alpha^{2}(\nu - L_{g}^{2}C_{P}^{2})||A w_{N}||^{2}_{\mathbb{L}^{2}}
			\end{aligned}
		\end{equation*}
		which is nonnegative when $L_{g} \leq \frac{\sqrt{\nu}}{C_{P}\sqrt{2}}$.
	\end{proof}
	\begin{rmk}
		The quantities $\bar{u}_{N}^{1}$ and $\bar{u}_{N}^{2}$ in the statement of Proposition~\ref{prop monotonicity} exist and are unique, thanks to the bijectivity of operator $I + \alpha^{2}A$ from $D(A)$ to $\mathbb{H}$.
	\end{rmk}
	\textbf{Limits identification:} For clarity's sake, the subsequences' subscript $N_{\ell}$ will be henceforth denoted $N$. Let $0 < m < N$ be a fixed integer, and $z, \bar{z} \in L^{\infty}(\Omega \times (0,T); V_{m})$ be such that $z = \bar{z} + \alpha^{2}A\bar{z}$. For $t \in [0,T]$, define the real valued process $\rho(\omega, t) \coloneqq \frac{2\mathcal{K}}{\nu^{3}}\int_{0}^{t}\left|\left|z(\omega, s)\right|\right|^{4}_{\mathbb{L}^{4}}ds$, where the constant $\mathcal{K}$ is that of Proposition~\ref{prop monotonicity}. Due to the properties of $z$, the process $\rho$ is clearly time-continuous and adapted. By application of It{\^o}'s formula to the process $t \mapsto e^{-\rho(t)}||v_{N}(t)||_{\mathbb{L}^{2}}^{2}$, it follows that
	\begin{equation*}
		\begin{aligned}
			&e^{-\rho(t)}||v_{N}(t)||^{2}_{\mathbb{L}^{2}} = ||v_{N}(0)||_{\mathbb{L}^{2}}^{2} + 	2\int_{0}^{t}e^{-\rho(s)}\left(v_{N}(s), g(s, \bar{u}_{N}(s))dW(s)\right) 
			\\&- \frac{2\mathcal{K}}{\nu^{3}}\int_{0}^{t}e^{-\rho(s)}||z(s)||^{4}_{\mathbb{L}^{4}}||v_{N}(s)||^{2}_{\mathbb{L}^{2}}ds - 2\int_{0}^{t}e^{-\rho(s)}\left(v_{N}(s), R(\bar{u}_{N}(s))\right)ds 
			\\&+ \int_{0}^{t}e^{-\rho(s)}||P_{N}g(s, \bar{u}_{N}(s))||^{2}_{\mathscr{L}_{2}(K, \mathbb{L}^{2})}ds,
		\end{aligned}
	\end{equation*}
	where we recall that $R(\bar{u}_{N}) = \nu Av_{N} + \tilde{B}(\bar{u}_{N}, v_{N})$. The mathematical expectation of the second term on the right-hand side is null, thanks to assumption~\ref{S2} and the measurability of $v_{N}$. Therefore, the above equation transforms into
	\begin{equation}\label{eq calc6}
		\begin{aligned}
			&\mathbb{E}\left[e^{-\rho(T)}||v_{N}(T)||^{2}_{\mathbb{L}^{2}} - ||v_{N}(0)||_{\mathbb{L}^{2}}^{2}\right] 
			\\&= \frac{2\mathcal{K}}{\nu^{3}}\mathbb{E}\left[\int_{0}^{T}e^{-\rho(t)}||z(t)||_{\mathbb{L}^{4}}^{4}\left\{||z(t)||_{\mathbb{L}^{2}}^{2} - 2\left(v_{N}(t), z(t)\right)\right\}dt\right] 
			\\& -2\mathbb{E}\left[\int_{0}^{T}e^{-\rho(t)}\left(R(\bar{u}_{N}(t)) - R(\bar{z}(t)) + \frac{\mathcal{K}}{\nu^{3}}||z(t)||^{4}_{\mathbb{L}^{4}}\left(v_{N}(t) - z(t)\right), v_{N}(t) - z(t)\right)dt\right] 
			\\& -2\mathbb{E}\left[\int_{0}^{T}e^{-\rho(t)}\left(R(\bar{u}_{N}(t)) - R(\bar{z}(t)), z(t)\right)dt\right] -2\mathbb{E}\left[\int_{0}^{T}e^{-\rho(t)}\left(R(\bar{z}(t)), v_{N}(t)\right)dt\right]
			\\&+\mathbb{E}\left[\int_{0}^{T}e^{-\rho(t)}||P_{N}g(t, \bar{u}_{N}(t)) - P_{N}g(t, \bar{z}(t))||^{2}_{\mathscr{L}_{2}(K, \mathbb{L}^{2})}dt\right] 
			\\&+2\mathbb{E}\left[\int_{0}^{T}e^{-\rho(t)}\left(P_{N}g(t, \bar{u}_{N}(t)), P_{N}g(t, \bar{z}(t))\right)_{\mathscr{L}_{2}(K, \mathbb{L}^{2})}dt\right]
			\\&- \mathbb{E}\left[\int_{0}^{T}e^{-\rho(t)}||P_{N}g(t, \bar{z}(t))||^{2}_{\mathscr{L}_{2}(K, \mathbb{L}^{2})}dt\right] \eqqcolon I_{1} + \dotsc + I_{7},
		\end{aligned}
	\end{equation}
	where the notation $\left(\cdot, \cdot\right)_{\mathscr{L}_{2}(K, \mathbb{L}^{2})}$ represents the $\mathscr{L}_{2}(K, \mathbb{L}^{2})$-scalar product. By convergence~\ref{appen conv1}, $I_{1}$ converges toward $\frac{2\mathcal{K}}{\nu^{3}}\mathbb{E}\left[\int_{0}^{T}e^{-\rho(t)}||z(t)||^{4}_{\mathbb{L}^{4}}\left\{||z(t)||^{2}_{\mathbb{L}^{2}} - 2\left(v_{NS}(t), z(t)\right)\right\}dt\right]$ as $N \to +\infty$. Moreover, 
	\begin{equation*}
		\begin{aligned}
			I_{2} =& -2\mathbb{E}\left[\int_{0}^{T}e^{-\rho(t)}\left(R(\bar{u}_{N}(t)) - R(\bar{z}(t)) + \frac{\mathcal{K}}{\nu^{3}}||z(t)||^{4}_{\mathbb{L}^{4}}\big(\bar{u}_{N}(t), \bar{z}(t)\big), \bar{u}_{N}(t) - \bar{z}(t)\right)dt\right] 
			\\&-2\alpha^{2}\mathbb{E}\left[\int_{0}^{T}e^{-\rho(t)}\left(R(\bar{u}_{N}(t)) - R(\bar{z}(t)) + \frac{\mathcal{K}}{\nu^{3}}||z(t)||^{4}_{\mathbb{L}^{4}}(\bar{u}_{N}(t) - \bar{z}(t)), A\bar{u}_{N}(t) - A\bar{z}(t)\right)dt\right]
			\\&-\frac{2\mathcal{K}\alpha^{2}}{\nu^{3}}\mathbb{E}\left[\int_{0}^{T}e^{-\rho(t)}||z(t)||^{4}_{\mathbb{L}^{4}}\left(A\bar{u}_{N}(t) - A\bar{z}(t), v_{N}(t) - z(t)\right)dt\right] \eqqcolon I_{2,1} + I_{2,2} + I_{2,3}.
		\end{aligned}
	\end{equation*}
	Proposition~\ref{prop monotonicity} implies that $I_{2,1} + I_{5} \leq 0$. Additionally, by turning $\left(\cdot , \cdot\right)$ into $\langle \cdot , \cdot \rangle$, it follows that
	\begin{equation*}
		|I_{2,2}| \leq 2\alpha^{2}\mathbb{E}\left[\int_{0}^{T}\left(||R(\bar{u}_{N})||_{\mathbb{H}^{-1}} + ||R(\bar{z})||_{\mathbb{H}^{-1}} + \frac{\mathcal{K}}{\nu^{3}}||z||^{4}_{\mathbb{L}^{4}}||\bar{u}_{N} - \bar{z}||_{\mathbb{H}^{-1}}\right)||A\bar{u}_{N} - A\bar{z}||_{\mathbb{H}^{1}}dt\right].
	\end{equation*}
	By the definition of operator $R$, one gets
	\begin{equation*}
		\begin{aligned}
			||R(\bar{u}_{N}(t))||_{\mathbb{H}^{-1}} &\leq \nu||A\bar{u}_{N}(t)||_{\mathbb{H}^{-1}} + ||\tilde{B}(\bar{u}_{N}(t), v_{N}(t))||_{\mathbb{H}^{-1}}
			\\& \leq \nu||\nabla\bar{u}_{N}(t)||_{\mathbb{L}^{2}} + C_{D}||\bar{u}_{N}(t)||_{\mathbb{L}^{2}}^{\frac{1}{2}}||\nabla\bar{u}_{N}(t)||_{\mathbb{L}^{2}}^{\frac{1}{2}}||\nabla v_{N}(t)||_{\mathbb{L}^{2}},
		\end{aligned}
	\end{equation*}
	thanks to Proposition~\ref{prop trilinear term}-\textit{(iii)} and the Gagliardo-Nirenberg inequality. Therefore,
	\begin{equation*}
		\begin{aligned}
			&2\alpha^{2}\mathbb{E}\left[\int_{0}^{T}||R(\bar{u}_{N}(t))||_{\mathbb{H}^{-1}}||A\bar{u}_{N}(t) - A\bar{z}(t)||_{\mathbb{H}^{1}}dt\right] \\&\leq 2\alpha\nu\mathbb{E}\left[\sup_{t \in [0,T]}||\nabla\bar{u}_{N}(t)||_{\mathbb{L}^{2}}^{2}\right]^{\frac{1}{2}}\mathbb{E}\left[\int_{0}^{T}\alpha^{2}||A\bar{u}_{N}(t) - A\bar{z}(t)||^{2}_{\mathbb{H}^{1}}dt\right]^{\frac{1}{2}}
			\\&\hspace{10pt}+ 2\alpha C_{D}\mathbb{E}\left[\sup_{t \in [0,T]}||\bar{u}_{N}(t)||_{\mathbb{L}^{2}}||\nabla\bar{u}_{N}(t)||_{\mathbb{L^{2}}}||\nabla v_{N}(t)||_{\mathbb{L}^{2}}^{2}\right]^{\frac{1}{2}}\mathbb{E}\left[\int_{0}^{T}\alpha^{2}||A\bar{u}_{N}(t) - A\bar{z}(t)||^{2}_{\mathbb{H}^{1}}dt\right]^{\frac{1}{2}}
			\\&\lesssim 2\alpha\nu C_{3} + 2\alpha C_{D}C_{3}C_{4} \to 0 \mbox{ as } N \to +\infty ,
		\end{aligned}
	\end{equation*}
	thanks to Lemma~\ref{appen lemma a priori estimates 2} and the assumption $\alpha \leq \mathcal{C}_{max}\mu_{N}^{-3/4}$. The same goes for the remaining terms of $I_{2,2}$, which are easier to handle. Thus, $I_{2,2} \to 0$ as $N \to +\infty$. Moving on to $I_{2,3}$, we have
	\begin{equation*}
		\begin{aligned}
			|I_{2,3}| &\leq \frac{2\mathcal{K}\alpha}{\nu^{3}}\mathbb{E}\left[\sup_{t \in [0,T]}||z(t)||^{8}_{\mathbb{L}^{4}}||v_{N}(t) - z(t)||^{2}_{\mathbb{L}^{2}}\right]^{\frac{1}{2}}\mathbb{E}\left[\int_{0}^{T}\alpha^{2}||A\bar{u}_{N}(t) - A\bar{z}(t)||^{2}_{\mathbb{L}^{2}}dt\right]^{\frac{1}{2}} 
			\\&\lesssim \frac{2\mathcal{K}\alpha}{\nu^{3}}C_{2}C_{1} \to 0 \mbox{ as } N \to +\infty,
		\end{aligned}
	\end{equation*}
	by virtue of Lemma~\ref{appen lemma a priori estimates 1} and $\alpha \leq \mathcal{C}_{max}\mu_{N}^{-3/4}$. It is straightforward to show that when $N \to +\infty$, $z$ and $\bar{z}$ become equal $\mathbb{P}$-a.s. and a.e. in $[0,T] \times D$. We exploit this fact and convergence~\ref{appen conv3} to obtain $I_{3} \to -2\mathbb{E}\left[\int_{0}^{T}e^{-\rho(t)}\langle R_{0}(t) - R(z(t)), z(t)\rangle dt\right]$ as $N \to +\infty$, and convergence~\ref{appen conv1} to accomplish $I_{4} \to -2\mathbb{E}\left[\int_{0}^{T}e^{-\rho(t)}\langle R(z(t)), v_{NS}(t)\rangle dt\right]$. Similarly, $I_{6} \to 2\mathbb{E}\left[\int_{0}^{T}e^{-\rho(t)}\left(g_{0}(t), g(t, z(t))\right)_{\mathscr{L}_{2}(K, \mathbb{L}^{2})}dt\right]$, thanks to result~\eqref{appen conv4}, the continuity of $g$ with respect to its second variable, and the properties of projector $P_{N}$ which also grant the convergence of $I_{7}$ i.e. $I_{7} \to -\mathbb{E}\left[\int_{0}^{T}e^{-\rho(t)}||g(t, z(t))||^{2}_{\mathscr{L}_{2}(K, \mathbb{L}^{2})}dt\right]$. Consequently, we pass to the limit in equation~\eqref{eq calc6} while taking advantage of all generated results to achieve eventually:
	\begin{equation}\label{eq calc7}
		\begin{aligned}
			&\mathbb{E}\left[e^{-\rho(T)}||v_{NS}(T)||^{2}_{\mathbb{L}^{2}} - ||v_{NS}(0)||^{2}_{\mathbb{L}^{2}}\right] \leq \liminf\limits_{N \to +\infty}\mathbb{E}\left[e^{-\rho(T)}||v_{N}(T)||^{2}_{\mathbb{L}^{2}} - ||v_{N}(0)||_{\mathbb{L}^{2}}^{2}\right] 
			\\&\leq \frac{2\mathcal{K}}{\nu^{3}}\mathbb{E}\left[\int_{0}^{T}e^{-\rho(t)}||z(t)||_{\mathbb{L}^{4}}^{4}\left\{||z(t)||^{2}_{\mathbb{L}^{2}} - 2\left(v_{NS}(t), z(t)\right)\right\}dt\right]
			\\&-2\mathbb{E}\left[\int_{0}^{T}e^{-\rho(t)}\langle R_{0}(t) - R(z(t)), z(t)\rangle dt\right] - 2\mathbb{E}\left[\int_{0}^{T}e^{-\rho(t)}\langle R(z(t)), v_{NS}(t)\rangle dt\right] 
			\\&+2\mathbb{E}\left[\int_{0}^{T}e^{-\rho(t)}\left(g_{0}(t), g(t, z(t))\right)_{\mathscr{L}_{2}(K, \mathbb{L}^{2})}dt\right] -\mathbb{E}\left[\int_{0}^{T}e^{-\rho(t)}||g(t, z(t))||^{2}_{\mathscr{L}_{2}(K, \mathbb{L}^{2})}dt\right].
		\end{aligned}
	\end{equation}
	Next, we apply It{\^o}'s formula to the process $t \mapsto e^{-\rho(t)}||v_{NS}(t)||^{2}_{\mathbb{L}^{2}}$, where we recall that $v_{NS}$ satisfies equation~\eqref{eq limiting equation}. It holds that
	\begin{equation}\label{eq calc8}
		\begin{aligned}
			&\mathbb{E}\left[e^{-\rho(T)}||v_{NS}(T)||_{\mathbb{L}^{2}}^{2} - ||v_{NS}(0)||_{\mathbb{L}^{2}}^{2}\right] = -\frac{2\mathcal{K}}{\nu^{3}}\mathbb{E}\left[\int_{0}^{T}e^{-\rho(t)}||z(t)||_{\mathbb{L}^{4}}^{4}||v_{NS}(t)||^{2}_{\mathbb{L}^{2}}dt\right]
			\\&-2\mathbb{E}\left[\int_{0}^{T}e^{-\rho(t)}\langle R_{0}(t), v_{NS}(t)\rangle dt\right] + \mathbb{E}\left[\int_{0}^{T}e^{-\rho(t)}||g_{0}(t)||^{2}_{\mathscr{L}_{2}(K, \mathbb{L}^{2})}dt\right].
		\end{aligned}
	\end{equation}
	Plugging result~\eqref{eq calc8} in equation~\eqref{eq calc7} grants:
	\begin{equation}\label{eq calc9}
		\begin{aligned}
			&\frac{2\mathcal{K}}{\nu^{3}}\mathbb{E}\left[\int_{0}^{T}e^{-\rho(t)}||z(t)||^{4}_{\mathbb{L}^{4}}||v_{NS}(t) - z(t)||^{2}_{\mathbb{L}^{2}}dt\right] 
			\\&+ 2\mathbb{E}\left[\int_{0}^{T}e^{-\rho(t)}\langle R_{0}(t) - R(z(t)), v_{NS}(t) - z~(t)\rangle dt\right]
			\\&\geq \mathbb{E}\left[\int_{0}^{T}e^{-\rho(t)}||g_{0}(t) - g(t, z(t))||^{2}_{\mathscr{L}_{2}(K, \mathbb{L}^{2})}dt\right], \ \ \forall z \in L^{\infty}(\Omega \times (0,T); V_{m}).
		\end{aligned}
	\end{equation}
	Arguing by density, the above inequality holds for all $z \in L^{4}(\Omega; L^{\infty}(0,T; \mathbb{H})) \cap L^{2}(\Omega; L^{2}(0,T; \mathbb{V}))$. Setting $z = v_{NS}$ in equation~\eqref{eq calc9} yields $g(\cdot, v_{NS}) = g_{0}$ $\mathbb{P}$-a.s. and a.e. in $(0,T) \times D$. Furthermore, for an arbitrary $w \in L^{4}(\Omega; L^{\infty}(0,T; \mathbb{H})) \cap L^{2}(\Omega; L^{2}(0,T; \mathbb{V}))$ and $\theta \in \mathbb{R}_{+}^{*}$, we set $z = v_{NS} + \theta w$, and make use of equation~\eqref{eq calc9} once again to obtain:
	\begin{equation*}
		\begin{aligned}
			&\frac{\mathcal{K}\theta}{\nu^{3}}\mathbb{E}\left[\int_{0}^{T}e^{-\rho(t)}||v_{NS}(t) + \theta w(t)||^{4}_{\mathbb{L}^{4}}||w(t)||^{2}_{\mathbb{L}^{2}}dt\right] 
			\\&- \mathbb{E}\left[\int_{0}^{T}e^{-\rho(t)}\langle R_{0}(t) - R(v_{NS}(t) + \theta w(t)), w(t)\rangle dt\right] \geq 0.
		\end{aligned}
	\end{equation*}
	Letting $\theta$ go to $0$ and using the hemi-continuity of the operator $R$ lead to 
	\begin{equation*}
		\mathbb{E}\left[e^{-\rho(t)}\langle R_{0}(t) - R(v_{NS}(t)), w(t)\rangle dt\right] \leq 0, \ \ \forall w \in L^{4}(\Omega; L^{\infty}(0,T; \mathbb{H})) \cap L^{2}(\Omega; L^{2}(0,T; \mathbb{V})),
	\end{equation*}
	which eventually implies $R_{0} = R(v_{NS})$ in $L^{2}(\Omega; L^{2}(0,T; \mathbb{H}^{-1}))$. The acquired limiting function $v_{NS}$ satisfies the following lemma.
	\begin{lem}\label{lemma a priori estimate v}
		Let $T > 0$, $1 \leq p < +\infty$, and $N \in \mathbb{N}\backslash \{0\}$ be given. Assume that hypotheses \ref{S1}-\ref{S2} are fulfilled, and that for some constant $\mathcal{C}_{max} > 0$ independent of $N$, the spatial scale $\alpha \leq \mathcal{C}_{max}\mu_{N}^{-3/4}$. Then, the process $\left\{v_{NS}(t), t \in [0,T]\right\}$ fulfills:
		\begin{enumerate}[label=(\roman*)]
			\item $\displaystyle \mathbb{E}\left[\sup_{t \in [0,T]}\left|\left|v_{NS}(t)\right|\right|^{2p}_{\mathbb{L}^{2}} + 2p\nu\int_{0}^{T}\left|\left|v_{NS}(t)\right|\right|^{2(p-1)}_{\mathbb{L}^{2}}\left|\left|\nabla v_{NS}(t)\right|\right|^{2}_{\mathbb{L}^{2}}dt\right] \leq C_{2}$,
			\item $\displaystyle \mathbb{E}\left[\sup_{t \in [0,T]}\left|\left|\nabla v_{NS}(t)\right|\right|_{\mathbb{L}^{2}}^{2p} + \left(\nu\int_{0}^{T}\left|\left|Av_{NS}(t)\right|\right|^{2}_{\mathbb{L}^{2}}\right)^{p}\right] \leq C_{4}$,
		\end{enumerate}
		where $C_{2} > 0$ depends on constants $\mathcal{C}_{max}$, $C_{1}$ of Lemma~\ref{appen lemma a priori estimates 1} and its parameters, and $C_{4} > 0$ depends on $C_{1}$, $||\bar{u}_{0}||_{L^{6p}(\Omega; \mathbb{V})}$ and $\mathcal{C}_{max}$.
	\end{lem}
	\begin{proof}
		We only illustrate here the proof of estimate \textit{(ii)} as \textit{(i)} can be concluded from \textit{(ii)}. Let $p \geq 1$. On account of Lemma~\ref{appen lemma a priori estimates 2}-\textit{(ii)}, the sequence $(v_{N})_{N}$ is bounded in $L^{2p}(\Omega; L^{\infty}(0,T; \mathbb{V}))$ which implies the existence of a function $\xi \in L^{2p}(\Omega; L^{\infty}(0,T; \mathbb{V}))$ such that for some subsequence $(v_{N_{\ell}})_{\ell}$, it holds that $v_{N_{\ell}} \overset{\ast}{\rightharpoonup} \xi$ in $L^{2p}(\Omega; L^{\infty}(0,T; \mathbb{V})) \hookrightarrow L^{2}(\Omega; L^{\infty}(0,T; \mathbb{H}))$, and 
		\begin{equation*}
			\mathbb{E}\left[\sup_{t \in [0,T]}||\xi(t)||^{2p}_{\mathbb{V}}\right] \leq \liminf\mathbb{E}\left[\sup_{t \in [0,T]}||v_{N_{\ell}}(t)||^{2p}_{\mathbb{V}}\right] \leq C_{4},
		\end{equation*}
		thanks to Lemma~\ref{appen lemma a priori estimates 2}-\textit{(ii)}. By convergence~\ref{appen conv2} and the weak limit uniqueness, we infer that $\xi = v_{NS}$ $\mathbb{P}$-a.s. and a.e. in $(0,T)\times D$. This is valid because $v_{NS}$ is the unique solution to equations~\eqref{eq NS} which means that the whole sequence $(v_{N})_{N}$ is convergent. Arguing in a similar fashion, and owing to Lemma~\ref{appen lemma a priori estimates 2}-\textit{(ii)}, $(v_{N})_{N}$ is bounded in the reflexive Banach space $L^{2p}(\Omega; L^{2}(0,T; D(A)))$, which signifies that for some $(v_{N_{\ell}})_{\ell}$ and $\eta \in L^{2p}(\Omega; L^{2}(0,T; D(A)))$, we have $v_{N_{\ell}} \rightharpoonup \eta$ in $L^{2p}(\Omega; L^{2}(0,T; D(A)))$ and 
		\begin{equation*}
			\mathbb{E}\left[\left(\nu\int_{0}^{T}||A\eta(t)||^{2}_{\mathbb{L}^{2}}dt\right)^{p}\right] \leq \liminf \mathbb{E}\left[\left(\nu\int_{0}^{T}||Av_{N_{\ell}}||^{2}_{\mathbb{L}^{2}}dt\right)^{p}\right] \leq C_{4}.
		\end{equation*}
		As done earlier in this proof, one obtains $\eta = v_{NS}$ $\mathbb{P}$-a.s. and a.e. in $(0,T) \times D$.
	\end{proof}
	
	\section{Conclusion}\label{section conclusion}
	Owing to Section~\ref{section convergence}, the limiting function $v_{NS}$ satisfies for all $t \in [0,T]$, and $\mathbb{P}$-a.s. the following equation:
	\begin{equation*}
		\begin{aligned}
			&\left(v_{NS}(t), \varphi\right) + \nu\int_{0}^{t}\left(\nabla v_{NS}(s), \nabla \varphi\right) + \int_{0}^{t}\tilde{b}(v_{NS}(s), v_{NS}(s), \varphi)ds 
			\\&= \left(v_{0}, \varphi\right) + \left(\int_{0}^{t}g(s, v_{NS}(s))dW(s), \varphi\right), \ \ \forall \varphi \in \mathbb{V}.
		\end{aligned}
	\end{equation*}
	By virtue of Proposition~\ref{prop trilinear term}, one gets 
	\begin{equation*}
		\tilde{b}(_{NS}(s), v_{NS}(s), \varphi) = - \tilde{b}(\varphi, v_{NS}(s), v_{NS}(s)) = \left([v_{NS}(s)\cdot\nabla]v_{NS}(s), \varphi\right).
	\end{equation*}
	Moreover, as mentioned in Section~\ref{section convergence}, $v_{NS}$ belongs to $L^{2}(\Omega; C([0,T]; \mathbb{H}))$. Besides the latter fact, Lemma~\ref{lemma a priori estimate v} guarantees that $v_{NS} \in L^{2}(\Omega; L^{2}(0,T; \mathbb{V}))$. Consequently, collecting all results and comparing them with Definition~\ref{definition NS solution}, it follows that $v_{NS}$ is the unique solution to equations~\eqref{eq NS} in the sense of Definition~\ref{definition NS solution}.
	
	The convergence analysis followed in this paper could have been carried out differently. For instance, instead of controlling the spatial scale $\alpha$ with a quantity that vanishes at the limit, a convergence rate of the difference $\left|\left|v_{NS}- \bar{u}\right|\right|$ in terms of $\alpha$ could have made up an alternative approach, as conducted in \cite{Bessaih2014Paul} for the stochastic Leray-$\alpha$ equations. We emphasize the uselessness of the imposed periodic boundary conditions if high spatial regularities of the solution were not utilized. In this case, Dirichlet boundary conditions are required.
	
	The demonstration techniques employed in this paper are only functional for two-dimensional domains. In three dimensions. another approach must be applied to acquire a solution to the stochastic Navier-Stokes problem from the stochastic LANS-$\alpha$ model, as performed in article~\cite{deugoue2011weak}.
	
	\printbibliography
\end{document}